\documentclass[12pt]{amsart}
\usepackage{amsmath,amsthm}
\usepackage{amsfonts}
\usepackage[psamsfonts]{amssymb}

\title[Modules of covariants]{Modules of covariants in modular invariant theory}
\author{Abraham Broer and Jianjun Chuai}
\address{ D\'{e}partement de math\'{e}matiques et de statistique
 \hfil\break\indent Universit\'{e} de Montr\'{e}al, Montr\'{e}al QC
H3C 3J7, Canada}
\email{broera@DMS.UMontreal.CA}

\address{ Department of Mathematics and Statistics
 \hfil\break\indent Memorial University of Newfoundland, St. John's NL
A1C 5S7, Canada}
\email{chuai@mun.ca}
\thanks{This research was supported by the Natural Sciences and Engineering
Research Council of Canada}

\date\today
\thanks{}

\theoremstyle{plain}
\newtheorem{lemma}{Lemma}
\newtheorem{proposition}{Proposition}
\newtheorem{theorem}{Theorem}

\newtheorem{corollary}{Corollary}

\theoremstyle{remark}
\newtheorem*{remark}{Remark}

\newtheorem{example}{Example} 

\def\F{{\mathbb F}}                            
\def\Z{{\mathbb Z}}                            

\def\codim{\mathop{\operatorname{codim}}\nolimits} 
\def\deg{\mathop{\operatorname{deg}}\nolimits} 
\def\det{{\mathop{\operatorname{det}}\nolimits}}
\def\dim{\mathop{\operatorname{dim}}\nolimits} 
\def\Diff{{\mathfrak D}} 
\def\Ind{\mathop{\operatorname{Ind}}\nolimits} 
\def\Jac{\mathop{\operatorname{Jac}}\nolimits}
\def\length{\mathop{\operatorname{\ell}}\nolimits} 
\def\normal{\vartriangleleft} 
\def\Ker{\mathop{\operatorname{Ker}}\nolimits} 
\def\GL{\mathop{\operatorname{GL}}\nolimits} 
\def\GF{{\mathcal{H}}} 
\def\End{\mathop{\operatorname{End}}\nolimits} 
\def\Tor{\mathop{\operatorname{Tor}}\nolimits} 
\def\Ext{\mathop{\operatorname{Ext}}\nolimits} 
\def\Hom{\mathop{\operatorname{Hom}}\nolimits} 
\def\Tr{\mathop{\operatorname{Tr}}\nolimits}   
\def\gotm{{\mathfrak m}} 

\def\gotM{{\mathfrak M}} 
\def\gotp{{\mathfrak p}} 
\def\gotP{{\mathfrak P}} 

\begin{document} 

\begin{abstract}
Let the finite group $G$ act linearly on the vector space $V$ over the field $k$ of
arbitrary characteristic, and let $H<G$ be a subgroup. The extension of invariant rings
$k[V]^G\subset k[V]^H$ is studied using modules of covariants. 

An example of our results is the following. Let $W$ be the subgroup of $G$ generated by the 
reflections in $G$.
A classical theorem due to Serre says that if $k[V]$ is a free $k[V]^G$-module then
$G=W$. We generalize this result as follows. If $k[V]^H$ is a free $k[V]^G$-module,      then
$G$ is generated by $H$ and $W$. Furthermore, the invariant ring $k[V]^{H\cap W}$ is free over $k[V]^W$
and is generated as an algebra by $H$-invariants and $W$-invariants.
\end{abstract}

\maketitle

\section*{Introduction} 
Let $V$ be an $n$-dimensional vector space over a field $k$ of arbitrary characteristic
and let $G$ be a finite group of linear automorphisms of $V$. Write $k[V]$ for the coordinate algebra of $V$; it is in a natural way a graded algebra.
 Then $G$ also acts as a group of graded $k$-algebra automorphisms of $k[V]$, and the collection of $G$-invariant
polynomials forms a graded $k$-algebra $k[V]^G$, the {\em algebra of invariants}. Let $G$ also act linearly on another finite dimensional vector space $M$ over the same
base field. Then we get an induced $G$-action on the free $k[V]$-module $k[V]\otimes_k M$.  The $G$-invariant elements of $k[V]\otimes_k M$ form a $k[V]^G$-submodule 
$$k[V]^G(M):=(k[V]\otimes_k M)^G,$$ 
called the {\em module of covariants of type $M$}.
The algebra of invariants  and the modules of covariants 
are classical objects of study. We shall say the situation is {\em modular} if the characteristic of the field 
$k$ is positive and divides the order of the group, and {\em non-modular} if it is not modular.

We shall say that
$\sigma\in G$ is a {\em reflection} if it fixes a linear hyperplane  of $V$ point-wise.
Put $W$ for the (normal) subgroup generated by the reflections in $G$.
This subgroup plays an important role.
The {\em Hilbert series} $\GF(k[V]^G(M);t)$ is defined as 
$$\GF(k[V]^G(M);t):=\sum_{i\geq 0} \dim_k(k[V]_i\otimes_k M)^G\cdot t^i.$$
Define $\deg(k[V]^G(M))$ and $\psi(k[V]^G(M))$ by the expansion of the Hilbert series at $t=1$:
$$\GF(k[V]^G(M);t)=\frac{\deg(k[V]^G(M))}{(1-t)^n}+
\frac{\psi(k[V]^G(M))}{(1-t)^{n-1}}+O(\frac{1}{(1-t)^{n-2}}).$$

In the non-modular situation, a lot more is known than in the modular 
situation.
For example, the following results are true in the non-modular situation but not so in the
modular situation.
$k[V]^G$ is a Cohen-Macaulay algebra and
$k[V]^G(M)$ is a Cohen-Macaulay module over $k[V]^G$, by Hochster-Eagon~\cite{HE}.  
The invariant algebra $k[V]^G$ is a polynomial algebra {\em if and only if} $G=W$, by Chevalley-Shephard-Todd~\cite[Theorem 7.2.1]{Benson1993}, and all modules of covariants are free when $G=W$.
There is a Molien formula for the Hilbert series $\GF(k[V]^G(M);t)$
using only the (Brauer) characters of $V$ and $M$, see e.g.~\cite{DerksenKemper}.

There are other results that are generally true, but whose proofs are much harder in the
modular case. For example, the following formula
\begin{equation}\label{eqnBC}
|G|\psi(k[V]^G)=\sum_{U\subset V} |G_U|\psi(k[V]^{G_U}),
\end{equation}
where the sum is over all linear hyperplanes $U\subset V$ and where $G_U$ is the
subgroup of $G$ consisting of all elements that fix $U$ point-wise. This follows easily
from Molien's formula in the non-modular situation, but the proof by 
Benson and Crawley-Boevey~\cite{BensonCB} requires a lot more work
in the modular situation.

If $k[V]$ is free as $k[V]^G$-module, then $G=W$. 
This result, first proved over the complex numbers by 
Shephard-Todd~\cite{ShephardTodd1954} and then in general by
Serre~\cite[V \S 5, ex.~8]{BourbakiLie}, \cite{Benson1993}, requires again a different approach in the modular case (see also~\cite{Br6}).

Stanley~\cite{Stanley} studied modules of covariants over the field of complex numbers when $M$ is one-dimensional. For example, he shows that then $k[V]^W(M)$ is  free of rank one
over $k[V]^W$, and he explicitly describes a generator $F_M\otimes v\in (k[V]\otimes M)^W$. Here $F_M\in k[V]$ is a certain product of linear forms and $v\in M$ is a basis element of the one-dimensional module $M$. 
Then he gives the criterion that $k[V]^G(M)$ is 
free over $k[V]^G$ {\em if and only if} $k[V]^W(M)$ and $k[V]^G(M)$ share generators {\em if and only
if} $k[V]^G(M)$ contains a non-zero homogeneous element of degree equal to the degree of $F_M$. Again, over
a general  field  these results remain true, but require additional work, see 
Nakajima~\cite{Na}.

In this article, we extend in some sense the above mentioned results of Benson and Crawley-Boevey, Serre, Stanley and Nakajima, to general modules of covariants over arbitrary fields. 

First, the following formula is a direct  analogue of (\ref{eqnBC}) for modules of covariants
\begin{equation}\label{eqnBCbis}
|G|\psi(k[V]^G(M))=\sum_{U\subset V} |G_U|\psi(k[V]^{G_U}(M)),
\end{equation}
where the sum is over all linear hyperplanes $U\subset V$, see Theorem~\ref{BCB} and Remark~\ref{psi}. Define $r_{k[V]^G}(k[V]^G(M))$, called the {\em rank}, and $s_{k[V]^G}(k[V]^G(M))$,
called the {\em $s$-invariant},  by the Taylor expansion at $t=1$ of the quotient of
Hilbert series
$$\frac{\GF(k[V]^G(M);t)}{\GF(k[V]^G;t)}=r_{k[V]^G}(k[V]^G(M))+s_{k[V]^G}(k[V]^G(M))(t-1)+
O((t-1)^2).$$
The rank is just the dimension of $M$.  If $k[V]^G(M)$ is a free $k[V]^G$-module 
with a homogeneous  basis of degrees $e_1,\ldots,e_m$, then the $s$-invariant coincides 
with the sum of these degrees. Formula (\ref{eqnBCbis}) is equivalent to the formula
\begin{equation}\label{sequation}
s_{k[V]^G}(k[V]^G(M))=\sum_{U\subset V,\ \codim_VU=1} s_{k[V]^{G_U}}(k[V]^{G_U}(M)).
\end{equation}
 We show in Theorem~\ref{BCB} that the $s$-invariant only depends on the reflections 
 in $G$, i.e.,
$$s_{k[V]^G}(k[V]^G(M))=s_{k[V]^W}(k[V]^W(M)).$$

Since $k[V]^{G_U}(M)$ is always a free graded $k[V]^{G_U}$-module, by Proposition~\ref{codimensionone}(i), 
it follows that the $s$-invariant of a module of covariants is always a non-negative integer.  Furthermore, the $s$-invariant is $0$ {\em if and only if} the subgroup $W$ acts
trivially on $M$, by Proposition~\ref{s is integer}. 

Next, we describe a generalization of a part of Stanley--Nakajima's results.
We prove in Theorem~\ref{tensor} that if  $k[V]^G(M)$ is free over $k[V]^G$, then $k[V]^W(M)$ is also free over $k[V]^W$ and both modules of covariants share generators. More precisely, 
 multiplication induces an isomorphism of $k[V]^W$-modules
 $$k[V]^W\otimes_{k[V]^G}k[V]^G(M)\simeq k[V]^W(M).$$
Conversely,  if $k[V]^W(M)$ is free over $k[V]^W$ and the modules of covariants $k[V]^W(M)$
and
$k[V]^G(M)$ share generators, then $k[V]^G(M)$ is free over $k[V]^G$, see Corollary~\ref{converse}.

We can use results on modules of covariants to study extensions of rings of invariants $k[V]^G\subset k[V]^H$, where $H<G$. Let $k(G/H)$ be the permutation module on $G/H$, the collection of the left cosets of $H$ in $G$. Then
the ring of invariants $k[V]^H$ is as a $k[V]^G$-module isomorphic to
the module of covariants of type $k(G/H)$, see Lemma~\ref{alsocovariants}.

We obtain a generalization of Serre's theorem, see Theorem~\ref{Serre}.
Suppose that $k[V]^H$ is free as a graded $k[V]^G$-module. Then  
$G$ is generated by $H$ together with the reflections in $G$,  i.e., $G=HW$. Furthermore, multiplication
induces an isomorphism of algebras
$$k[V]^H\otimes_{k[V]^G}k[V]^W\simeq k[V]^{H\cap W}.$$
In particular, in that case $k[V]^{W}\subset k[V]^{H\cap W}$ is also a free extension and
the $H$-invariants together with the $W$-invariants generate the ring of $H\cap W$-invariants, i.e.
$$k[V]^{H\cap W}=k[V]^Hk[V]^W.$$ 
Conversely, if $k[V]^{W}\subset k[V]^{H\cap W}$ is  a free extension and
the $H$-invariants together with the $W$-invariants generate the ring of $H\cap W$-invariants, we have that 
$k[V]^H$ is free as a graded $k[V]^G$-module, by Proposition~\ref{Serreconverse}.

For a further extension of Stanley--Nakajima's results,
we first define an analogue of Stanley's polynomial $F_M$  for a general $M$,  introduced by Gutkin~\cite{Gutkin} over the complex numbers.
It is an explicit product of linear forms, defined as follows.
For any linear hyperplane $U$ of $V$, fix a linear form $x_U$ having $U$ as zero-set. As before, 
we denote the point stabilizer of $U$ by $G_U$.
Put 
$$s_U(M):=s_{k[V]^{G_U}}(k[V]^{G_U}(M)).$$ 
Then we define
$$F_M:=\prod_{U\subset V} x_U^{s_U(M)},$$
where the product is over all the linear hyperplanes $U$ of $V$.  It is well-defined up to a non-zero scalar, since $s_U(M)\neq 0$ for only finitely many $U$'s.  From equation (\ref{sequation}) it follows that the degree of $F_M$ is exactly the $s$-invariant
of $k[V]^G(M)$.
Next, we define certain determinants.
Fix a basis $v_1,\ldots,v_m$ for $M$, where $m=\dim_k M$. 
If $\omega_j:=\sum_i f_{ij}\otimes v_i$, $1\leq j\leq m$, is an $m$-tuple of elements
in $k[V]^G(M)$, write
$$\Jac_M(\omega_1,\ldots,\omega_m)\in k[V]$$ for the determinant of the square $m\times m$ matrix
$(f_{ij})_{1\leq i,j\leq m}$.   We call it the {\em $M$-Jacobian determinant} of the $m$-tuple of covariants.

Now we are ready to state another generalization of Stanley-Nakajima's  freeness criterion,
called the {\em Jacobian criterion of freeness},
see Theorem~\ref{Jacobian}. 
Let $\omega_1,\ldots,\omega_m\in k[V]^G(M)$ be an $m$-tuple of homogeneous covariants of type $M$. They 
form a  free generating set of $k[V]^G(M)$ if and only
there is a non-zero constant $c\in k^\times$ such that
$\Jac_M(\omega_1,\ldots,\omega_m)=cF_M$ {\em if and only if} the
$M$-Jacobian determinant of $(\omega_1,\ldots,\omega_m)$ is non-zero and
the sum of the degrees of the $\omega_i$'s equals the $s$-invariant of $k[V]^G(M)$.

In this introduction, we described our results on modules of covariants of linear actions on a vector space.
But actions on polynomial rings respecting non-standard gradings are also interesting,
and likewise the action of $G/N$ on the invariant ring $k[V]^N$ of a normal subgroup $N\normal G$. That is
one of our motivations to work with actions of a finite group on an integrally closed, connected graded algebra without zero-divisors.  Many of the results mentioned in this introduction remain true in this more general
context. These assumptions imply that all modules of covariants are reflexive. This
is a property for modules analogous to the property of normality for an integral domain. 
A consequence, that we use frequently,  is that a morphism between two reflexive modules is an isomorphism if
and only if it is an isomorphism at all prime ideals of height one.

\section{Preliminaries and the $s$-invariant}
\subsection{Notation}\label{notation}
Fix a base field $k$ of arbitrary characteristic.
In this article,  all graded algebras $A=\bigoplus_{i\geq o} A_i$ will be assumed to be finitely generated graded algebras over the
field $k$ without zero-divisors,  and the only elements of degree
zero are the scalars, i.e., $A_0=k$. The elements without constant term form the unique maximal homogeneous ideal $A_+$.  Consequently, we have Nakayama's lemma for graded modules at our disposal, and so finitely generated projective graded modules are free. We shall say that $A$ is a {\em normal graded algebra} if it is a graded algebra 
integrally closed in its quotient field. A graded $A$-module $N=\oplus_{i\in\Z}N_i$ will always be assumed to be finitely generated. The {\em Hilbert series}
$$\GF(N;t):=\sum_{i} \dim_k(N_i)\ t^i$$
is then the Laurent expansion at $t=0$ of a rational function. For any integer $d$,  the {\em shifted graded module}
$N[d]$ is defined by $N[d]_i:=N_{d+i}$. So 
$$\GF(N[d];t)=\GF(N;t)\cdot t^{-d}.$$

Let $G$ be a finite group of graded $k$-algebra automorphisms on $A$, and $H<G$ a subgroup. 
We shall call  (the collection of cosets) $G/H$ {\em modular} if the characteristic of the base field is positive and divides $|G/H|=|G|/|H|$, and
{\em non-modular} otherwise. We shall be interested in the extension of graded algebras $A^G\subset A^H$, where $A^G$ and $A^H$
are the rings of invariants. It will be called a {\em free (graded) extension} if $A^H$ is a free graded module
over $A^G$; a {\em graded Gorenstein extension} if it is a free graded extension and the fibre algebra
$A^H/A^G_+A^H$ is a Gorenstein algebra (of Krull dimension zero); and a {\em graded complete intersection extension} 
if it is a free graded extension and the fibre algebra  is a graded complete intersection algebra. We refer to
\cite{Br6} for equivalent definitions and basic properties. For example, if $A^G\subset A^H$  is a free extension,
then one of the two invariant rings is Cohen-Macaulay {\em if and only if} the other is. If it is a graded
Gorenstein (resp. graded complete intersection) extension, then one of the two invariant rings is Gorenstein (resp. a complete intersection) {\em if and only if} the other is. Also several numerical
invariants are shared, see Avramov ~\cite{Av}. 

In most of the following definitions, we no longer need to assume that the $k$-algebra $A$ is graded.
The {\em inertia subgroup} $G_i(\gotP)$ of a prime ideal $\gotP\subset A$ is defined to be
$$G_i(\gotP):=\{\sigma\in G; \forall a\in A:\ \sigma(a)-a\in\gotP\};$$
it is a normal subgroup of the {\em decomposition group} 
$$G_d(\gotP):=\{\sigma\in G; \sigma(\gotP)=\gotP\}.$$
An element $\sigma\in G$ will be called a {\em reflection on $A$} if it
is contained in the inertia subgroup of some prime ideal of height one. The subgroup generated by all reflections on $A$ is denoted by $W$. It is a
normal subgroup of $G$. We shall say that $W$ is the {\em reflection subgroup} of $G$ acting on $A$.

The group algebra of $G$ over $k$ is denoted by $kG$.
Let $M$ be a finite dimensional $kG$-module. We call 
$$A^G(M):=(A\otimes_kM)^G$$ 
the {\em module of covariants of type $M$}. It is a finitely generated graded $A^G$-module
isomorphic to $\Hom_{kG}(M^*,A)$, where $M^*=\Hom_k(M,k)$ is the $kG$-module dual to $M$. 
If $\lambda:G\to k^\times$ is a linear character, i.e., a group homomorphism,
then 
$$A^G_\lambda:=\{a\in A;\forall \sigma\in G:\ \sigma(a)=\lambda(\sigma)a\}$$ is called the {\em module of semi-invariants
of type $\lambda$}. If $k_\lambda$ is the one dimensional $kG$-module on which $G$ acts by the
linear character $\lambda$ and $k^*_\lambda$ its dual, then 
$$A^G(k^*_\lambda)\simeq A^G_\lambda.$$

We denote the collection of all group homomorphisms  $\lambda:G\to k^\times$ by
$X(G)$; of course, it depends on the base field $k$. We denote the intersection of the kernels of all $\lambda \in X(G)$ by $G^1$. It contains
the derived group $G'$; and $G^1=G'$ if $k$ contains sufficiently many roots of unity.
The collection of semi-invariants for all types spans the subalgebra of $A^{G^1}$, and we have a finite direct sum of $A^G$-modules
$$A^{G^1}=\bigoplus_{\lambda\in X(G)} A^G_\lambda.$$

Let $G$ be a finite group of linear automorphisms of the vector space $V$ over $k$. The
coordinate ring $k[V]$ is the polynomial algebra $k[x_1,\ldots,x_n]$ on a basis of $V^*$. This is
a  graded algebra in the standard way, with $G$ as finite group of automorphisms. It is
our standard example of graded algebra. 
Let $\gotP\subset k[V]$ be a prime ideal of height one with non-trivial inertia subgroup. Then $\gotP=(f)$
for some irreducible polynomial $f\in k[x_1,\ldots,x_n]$, and there is a non-trivial $\sigma\in G$ such that
for all $i$ we have $\sigma(x_i)-x_i\in (f)$ and $\sigma(x_i)\neq x_i$ for at least one $i$. It follows
that $\gotP$ is generated by a linear form, say $x$. Then the decomposition subgroup $G_d(\gotP)$ is the subgroup of $G$ of elements stabilizing the zero set of $x$, and the inertia subgroup $G_i(\gotP)$ is the subgroup of $G$ fixing all elements  of the zero set of $x$.
So an element $\sigma\in G$ is a reflection on $k[V]$ {\em if and only if} $\sigma$ fixes point-wise some hyperplane of $V$.
The free module  $k[V]\otimes_k \wedge^iV^*$ can be identified
with the $k[V]$-module of polynomial differential $i$-forms on $V$, and $k[V]\otimes_k V$ with
the module of polynomial vector fields on $V$. So $k[V]^G(\wedge^iV^*)$ can be identified
with the $k[V]^G$-module of $G$-invariant polynomial $i$-forms, and $k[V]^G(V)$ with the module of $G$-invariant polynomial vector fields on $V$.

\subsection{The $s$-invariant}
Let $B$ be a graded algebra of Krull dimension $n$, and $N$ a graded $B$-module. Then
the numerical invariants $\deg(N)$ and $\psi(N)$ are defined by the Laurent expansion
of the Hilbert series of $N$ at $t=1$:
$$\GF(N;t)=\frac{\deg(N)}{(1-t)^n}+\frac{\psi(N)}{(1-t)^{n-1}}+O\left(\frac{1}{(1-t)^{n-2}}\right).$$
For some of the basic properties, see Benson~\cite{Benson1993}.
Closely related are the numerical invariants $r_B(N)$, called the {\em rank}, and $s_B(N)$, 
called the {\em $s$-invariant}, defined
by the Laurent expansion at $t=1$:
$$\frac{\GF(N;t)}{\GF(B;t)}=r_B(N)+s_B(N)(t-1)+O((t-1)^2).$$
Here $r_B(N)$ coincides with the usual rank of $N$ as $B$-module. 
The $s$-invariant was introduced into invariant theory by Brion~\cite{Brion} (but with a difference of sign). We refer 
to that article for some of the basic properties. For example, 
$$s_B(N[d])=s_B(N)-dr_B(N),$$
where $N[d]$ is the shifted graded module.
If $N$ is a free graded $B$-module of rank $r$ with homogeneous generators of degree $e_i$, then 
$$N\simeq \oplus_{i=1}^r B[-e_i];$$ 
and so we get the useful formula 
$$s_B(N)=\sum_{i=1}^r e_i.$$
If $\gotP\subset B$ is a homogeneous prime ideal of height at least one, then $r_B(B/\gotP)=0$ and
$$s_B(B/\gotP)=-\frac{\psi(B/\gotP)}{\deg B}.$$

The relationship between these numerical invariants is as follows.
\begin{lemma} \label{degpsi}
Let $B$ be a graded algebra, $N$ a  graded $B$-module and
$\rho: B_1\to B$ a homomorphism of graded algebras.

(i) We have $\deg N=r_B(N)\deg B$ and $\psi(N)=r_B(N)\psi(B)-s_B(N)\deg(B)$.

(ii) We have $r_{B_1}(N)=r_B(N)r_{B_1}(B)$ and
$s_{B_1}(N)=s_B(N)r_{B_1}(B)+r_{B}(N)s_{B_1}(B)$.
\end{lemma}

\begin{proof} 
This follows easily by comparing the Laurent series expansions at $t=1$.
\end{proof}

A homomorphism of graded $B$-modules $\phi:M\to N$ is called {\em pseudo-injective}
if the localization of the kernel vanishes at all homogeneous prime ideals of height at most one. Or, using the numerical
invariants, $\phi$ is pseudo-injective {\em if and only if}
$$r_B(\Ker\phi)=s_B(\Ker\phi)=0.$$
The notions
 {\em pseudo-surjective} and {\em pseudo-isomorphism} are defined similarly.

Let $M$ and $N$ be two graded $B$-modules. Then $\Hom_B(M,N)$ is also a graded $B$-module.
Benson and Crawley-Boevey~\cite{Benson1993} expressed the $\psi$-invariant of this
module in terms of the numerical invariants of $N,M$ and $B$. We write their result in terms of the $s$-invariants.

\begin{proposition}\label{Benson1}
Let $B$ be a normal graded domain and $M$ and $N$ two graded $B$-modules.

(i)
Then
$$s_B(\Hom_B(M,N))-s_B(\Ext^1_B(M,N))=r_B(M)s_B(N)-r_B(N)s_B(M).$$
If $M$ is torsion free, then $s_B(\Ext^1_B(M,N))=0$, and 
$$s_B(\Hom_B(M,B))=-s_B(M).$$

(ii) And
$$s_B(M\otimes_BN)-s_B(\Tor^B_1(M,N))=r_B(M)s_B(N)+r_B(N)s_B(M).$$
\end{proposition}

\begin{proof} 
(i) This is just the formula of Benson and Crawley-Boevey~\cite[Theorem 3.3.2]{Benson1993} reformulated in terms of the rank and the
$s$-invariant. Part (ii) is proved similarly as the cited formula for (i).
\end{proof}

The following result also follows from Benson and Crawley-Boevey~\cite{Benson1993}.
Recall that $W$ is the reflection subgroup of $G$ acting on $A$.
\begin{proposition}\label{Benson2}
Let $A$ be a graded normal domain with a finite group $G$ of automorphisms.

(i) Let $K$ be a subgroup such that  $W\leq K\leq G$.
Then 
$$s_{A^G}(A^K)=0.$$

(ii) Suppose $G_i(\gotP)\cap G_i(\gotP')=1$ for all distinct homogeneous height one
prime ideals $\gotP$ and $\gotP'$ of $A$. Then
$$\frac{1}{|G|}s_{A^G}(A)=\sum_{\gotP}\frac{1}{|G_i(\gotP)|}s_{A^{G_i(\gotP)}}(A),$$
where the sum is over all homogeneous prime ideals of height one.

(iii) Suppose $A$ is factorial. Put $\delta_G$ for the degree of a generator of the
Dedekind different $\Diff_{A/A^G}$, and define similarly $\delta_H$ for a subgroup $H<G$. Then
 $$s_{A^G}(A^H)=\frac{|G|}{2|H|}(\delta_G-\delta_H).$$
\end{proposition}

\begin{proof}
(i) Since the Dedekind different of the extension $A^G\subset A^K$ is trivial, it follows from
\cite[Theorem 3.12.1]{Benson1993} that 
$$|G|\psi(A^G)=|K|\psi(A^K).$$ 
Or in terms of the
$s$-invariant,
$$s_{A^G}(A^K)=s_{A^K}(A^K)=0.$$

(ii) This is a reformulation in terms  of the $s$-invariant of \cite[Corollary 3.12.2]{Benson1993} using that
$$|G|\psi(A^G)-\psi(A)=\frac{\deg A}{|G|}s_{A^G}(A).$$

(iii) Again from \cite[Theorem 3.12.1]{Benson1993}  it follows that
$$|G|\psi(A^G)-\psi(A)=\frac{\delta_G \deg A}{2}.$$ 
Combining with the formula above we get
$$s_{A^G}(A)=\frac{|G|\delta_G}{2}.$$ 
Similarly with $G$ replaced by $H$. Now applying Lemma~\ref{degpsi}(ii) gives the result.
\end{proof}

We partially generalize this result to modules of covariants as follows.  A proof is given in
Section~\ref{ramification}.

\begin{theorem}\label{BCB}
Let $A$ be a graded normal algebra with finite group $G$ of graded $k$-algebra automorphisms, and $M$ a finite dimensional $kG$-module.

(i)
Let $K$ be a subgroup such that $W\leq K\leq G$. Then
$$s_{A^G}(A^G(M))=s_{A^K}(A^K(M)).$$
In particular, if $W$ acts trivially on $M$, then 
$s_{A^G}(A^G(M))=0.$

(ii)
Suppose $G_i(\gotP)\cap G_i(\gotP')=1$ for all distinct homogeneous height one
prime ideals $\gotP$ and $\gotP'$ of $A$. Then 
$$s_{A^G}(A^G(M))=
\sum_{\gotP}s_{A^{G_i(\gotP)}}(A^{G_i(\gotP)}(M)),$$
where the sum is over the homogeneous height one prime ideals of $A$.
\end{theorem}

\begin{proof} 
See Corollary~\ref{newBenson} in the last section.
\end{proof}

In the linear case the $s$-invariant of a module of covariants is always a non-negative integer.

\begin{proposition}\label{s is integer}
Let $G$ be a finite group of linear automorphisms of the vector space $V$ and $M$ a finite dimensional $kG$-module. 
Then $s_{k[V]^G}(k[V]^G(M))$ is a non-negative integer and $s_{k[V]^G}(k[V]^G(M))=0$  {\em if and only} if the reflection subgroup $W$ acts trivially
on $M$. 
\end{proposition}

\begin{proof}
We start with a special case.
If $k[V]^G(M)$ is free over $k[V]^G$ with homogeneous generators $\omega_1,\ldots,\omega_m$,
then the $s$-invariant of $k[V]^G(M)$ is $\sum_i\deg(\omega_i)$. Since $k[V]^G(M)$ has no elements of negative degree it follows that the $s$-invariant
is a non-negative integer. If the $s$-invariant is $0$, then each generator is of degree $0$ and hence
of the form 
$$\omega_i =1\otimes u_i\in (k[V]\otimes M)^G.$$ 
Here each $u_i\in M^G$, and
$u_1,\ldots,u_m$ form a basis of $M$. It follows that $G$ acts trivially on $M$.

Now we no longer assume freeness.
The conditions of Theorem~\ref{BCB}(ii) are satisfied for a linear action. By Proposition~\ref{codimensionone}(i) or Hartmann--Shepler~\cite{HartmannS}, for every homogeneous prime ideal $\gotP\subset k[V]$ of height one
the module of covariants $k[V]^{G_i(\gotP)}(M)$ is free and so by the special case above,
$s_{k[V]^{G_i(\gotP)}}(k[V]^{G_i(\gotP)}(M))\geq 0$ and 
$s_{k[V]^{G_i(\gotP)}}(k[V]^{G_i(\gotP)}(M))=0$ {\em if and only if} $G_i(\gotP)$ acts
trivially on $M$. Since the $G_i(\gotP)$'s generate $W$, applying Theorem~\ref{BCB} gives the result.
\end{proof}

A short exact sequence of $kG$-modules 
$$0\to M_1\to M_2\to M_3\to 0$$ 
gives rise
to a left  exact sequence of modules of covariants
$$0\to A^G(M_1)\to A^G(M_2)\to A^G(M_3).$$

\begin{lemma}
Let 
$$0\to M_1\to M_2\to M_3\to 0$$ 
be a short exact sequence of finite dimensional $kG$-modules. Put $K$ for the
kernel of the action of $G$ on $M_2$. If the relative trace ideal $\Tr^G_K(A^K)\subseteq A^G$ has
height bigger than $1$, then
$$s_{A^G}(A^G(M_2))=s_{A^G}(A^G(M_1))+s_{A^G}(A^G(M_3)).$$
Here $\Tr^G_K(a)=\sum_{\sigma K\in G/K}\sigma(a),$ for $a\in A^K$.
\end{lemma}

\begin{proof}
The normal subgroup $K$ also acts trivially on $M_1$ and $M_3$. So 
$$(A\otimes_k M_i)^K\simeq
A^K\otimes M_i$$ 
and we get a short
exact sequence
$$0\to A^K\otimes_kM_1\to A^K\otimes_kM_2\to A^K\otimes_kM_3\to 0$$
of $G/K$-modules.
From the theory of group cohomology we get a long exact sequence of graded $A^G=(A^K)^{G/K}$-modules 
$$0\to A^G(M_1)\to A^G(M_2)\to A^G(M_3)\to H^1(G/K, A^K\otimes M_1)\to H^1(G/K,A^K\otimes M_2)\ldots.$$
It is known that all $H^i(G/K,A^K\otimes M_j)$, $i\geq 1$, are annihilated by the ideal $\Tr^G_K(A^K)\subseteq A^G$, e.g.~\cite{DerksenKemper} or \cite[Lemma 1.3]{LP}.
So the cokernel of 
$$A^G(M_2)\to A^G(M_3)$$ 
is annhilated by $\Tr^G_K(A^K)$ as well. Our
hypothesis implies therefore that the  support of the cokernel of 
$$A^G(M_2)\to A^G(M_3)$$
is of codimension $\geq 2$, and hence is pseudo-zero with vanishing $s$-invariant. 
\end{proof}

\begin{example}\label{firstexample}
Let $k=\F_2$, $V=\F_2^4$, and $G$ the abelian group generated by the three reflections
$$\sigma_1:=\left(\begin{matrix}1&0&0&0\\0&1&0&0\\0&0&1&0\\1&0&0&1\\\end{matrix}\right),\
\sigma_2:=\left(\begin{matrix}1&0&0&0\\0&1&0&0\\0&1&1&0\\0&0&0&1\\\end{matrix}\right),\ 
\sigma_3:=\left(\begin{matrix}1&0&0&0\\0&1&0&0\\1&1&1&0\\1&1&0&1\\\end{matrix}\right).$$
Let $x_1,x_2,x_3,x_4$ be the basis of $V^*$ dual to the standard basis of $V$. Then
$k[V]$ is the polynomial ring $\F_2[x_1,x_2,x_3,x_4]$. The only homogeneous prime ideals of height one in $k[V]$ with non-trivial
inertia subgroups are the principal ideals $\gotP_1=(x_1)$, $\gotP_2=(x_2)$ and $\gotP_3=(x_1+x_2)$. The
inertia subgroup $H_i$ of $\gotP_i$ is $\{1,\sigma_i\},$ for $i=1$, $2$ or $3$.
Let 
$$M_1:=\F_2 x_1+\F_2x_2+\F_2x_3.$$ 
Then $M_1$ is a $kG$-submodule of $M_2:=V^*$, and the
quotient module $M_3=M_2/M_1$ is the trivial $kG$-module. 

Free generators of $(k[V]\otimes M_1)^{H_i}$ are $1\otimes x_1$, $1\otimes x_2$,
together with  $1\otimes x_3$ if $i=1$, or with $x_2\otimes x_3+x_3\otimes x_2$ if $i=2$,
or with $(x_1+x_2)\otimes x_3+x_3\otimes(x_1+x_2)$ if $i=3$.
So 
$$s_{k[V]^G}(k[V]^G(M_1))=0+1+1=2.$$

Free generators of $(k[V]\otimes M_2)^{H_i}$ are $1\otimes x_1$, $1\otimes x_2$,
together with  $1\otimes x_3$ and $x_1\otimes x_4+x_4\otimes x_1$ if $i=1$, or with
$x_2\otimes x_3+x_3\otimes x_2$ and  $1\otimes x_4$ if $i=2$,
or with $(x_1+x_2)\otimes x_3+x_3\otimes(x_1+x_2)$
and $1\otimes(x_3+x_4)$ if $i=3$.
So 
$$s_{k[V]^G}(k[V]^G(M_2))=1+1+1=3.$$

Since $M_3$ is the trivial $kG$-module, we have 
$$s_{k[V]^G}(k[V]^G(M_3))=0.$$
Indeed, in this case 
$$s_{A^G}(A^G(M_2))\neq s_{A^G}(A^G(M_1))+s_{A^G}(A^G(M_3)).$$ 
So the ideal
$\Tr^G(k[V])\subset k[V]^G$ must be of height one.
\end{example}

\begin{example}\label{secondexample}
Let $k=\F_p$ be the prime field of order $p$, $V$ an $n$ dimensional vector space over $\F_p$, and $G$ a  group of $\F_p$-linear
automorphisms of $V$. Let $U\subset V$ be a linear hyperplane defined by the linear form $x_U$, and let $G_U$ be its point stabilizer.
It has a semi-direct product decomposition 
$$G_U=(\Z/p\Z)^{a_U}\rtimes \Z/h_U\Z,$$ 
where
$a_U\leq n-1$ and $h_U$ divides $p-1$. Let $\zeta$ be a primitive $h_U$-th root of unity in $\F_p$.
We can choose generators $\tau$ and
$\sigma_1,\ldots,\sigma_{a_U}$ of $G_U$ and a basis $x_1=x_U,x_2,\ldots,x_n$ of coordinate functions,
such that 
$$\tau x_1=\zeta x_1,\ \sigma_i x_{i+1}=x_{i+1}+x_1,\ 1\leq i \leq a_U,$$
 and all
other actions are trivial. Then a free basis for $(\F_p[V]\otimes V^*)^{G_U}$ is
$$x_1^{h_U-1}\otimes x_1,\ x_{i+1}\otimes x_1-x_1\otimes x_{i+1},\ {\rm if}\ 
1\leq i\leq a_U,\ {\rm and}\
1\otimes x_i\ {\rm if}\ i>a_{U}+1.$$ 
Then 
$$s_{k[V]^{G_U}}(k[V]^{G_U}(V^*))=(h_U-1+a_U)$$ 
and so
$$s_{\F_p[V]^G}(\F_p[V]^G(V^*))=
\sum_{U\subset V}(h_U-1+a_U),$$
 where the sum is over the linear hyperplanes of $V$. 
On the other hand, Benson  and Crawley-Boevey \cite[Theorem 3.13.2]{Benson1993} calculated that
$$s_{\F_p[V]^G}(\F_p[V])=\frac{|G|}{2}\delta,$$ where
$$\delta=\sum_{U\subset V}(h_U-1+(p-1)a_U)$$
is the degree of a generator of the Dedekind different $\Diff_{k[V]/k[V]^G}$.
We conclude that if either $p=2$ or $G$ is transvection free, then 
$$s_{\F_p[V]^G}(\F_p[V]^G(V^*))=\delta.$$
\end{example}

\section{Reflexive modules of covariants}
Let $A$ be a normal graded algebra, and $A^G$ the invariant ring of a finite group action. A well-known fundamental fact is that  $A^G$ is also a normal graded algebra, see~\cite{Benson1993}.
Less well-known is that  all modules of covariants $A^G(M)=(A\otimes_k M)^G$ 
inherit an analogous good property, namely that they are reflexive finitely
generated graded modules, see e.g. Brion~\cite{Brion}. In this section we shall reprove
this property, and use the fundamental property that a quasi-isomorphism between
reflexive modules is already an isomorphism.  

We recall that a finitely generated module $M$ over a ring $R$ is called {\em reflexive} if the natural map
$M\to M^{**}$  is an isomorphism, where $M^*=\Hom_R(M,R)$ is the dual module.
For example, free modules are reflexive and reflexive modules are torsion free.
For more information, see \cite[VII \S 4]{Bourbaki} or \cite[\S 3.4]{Benson1993}.
That reflexivity is the module analogue of normality is shown clearly in the next lemma.

\begin{lemma}\label{classic}
Let $B$ be a normal graded algebra with quotient field $K$ and $N$ a finitely generated torsion free graded module. Then the following are equivalent.

(i) $N$ is reflexive;

(ii) $N=\cap_{\gotp}N_\gotp$, where the intersection is inside
the vector space $N\otimes_BK$ running over all prime ideals $\gotp\subset B$ of height one ;

(iii) Every regular sequence of length  two on $B$ is also a regular sequence of length two
on $N$.
\end{lemma}

\begin{proof}
See Samuel~\cite[Proposition 1]{Samuel}. In the proof it is not needed that $B$ is graded.
\end{proof}

In the applications, we shall mainly use the following properties of reflexive modules.

\begin{lemma}\label{reflexive}
Let $B$ be a normal graded algebra and $M$ and $N$ two reflexive graded $B$-modules.

(i) If $\phi:M\to N$ is a pseudo-isomorphism, then it is an isomorphism.

(ii) The $B$-module $\Hom_B(M,N)$ is also a reflexive graded module.

(iii) If $B\subset A$ is a finite free graded extension, then $A\otimes_BM$ is a reflexive
$A$-module.

(iv) Let $B\subset A$ be a finite extension of normal graded algebras whose
extension of quotient fields is separable.  A graded $A$-module
is reflexive as $A$-module {\em if and only if} it is reflexive considered as $B$-module by restriction.
\end{lemma}

\begin{proof}
(i) Let $K$ be the quotient field of $B$. Since the ranks of the kernel and the cokernel are zero, it follows
that 
$$K\otimes_BM\simeq K\otimes_BN.$$ 
Identify both $K$-vector spaces and write $V=K\otimes_BM$. Since $M$ is reflexive,
$$M=\cap_{\gotP} M_\gotP,$$ 
where the intersection is taken inside $V$ and $\gotP$ runs over all height one
prime ideals, cf. \cite[VII \S 4.2 Th\'eor\`eme 2]{Bourbaki}. The same for $N$. Since $\phi$ is a pseudo-isomorphism we have, for each
height one prime ideal $\gotP$ in $B$, that $M_\gotP=N_\gotP\subset V$. Hence $M=N$
and $\phi$ is an isomorphism.

For (ii), use \cite[Lemma 3.4.1(v)]{Benson1993} and for (iii), use \cite[VII \S 4.2 Proposition 8]{Bourbaki}.

(iv) Let $M$ be a graded $A$-module. For any height one prime ideal $\gotp\subset B$ let $\gotP_1,\ldots,\gotP_t$ be
the prime ideals in $A$ lying over $\gotp$. Then all $\gotP_i$'s are of height one. It suffices to prove
that 
$$M_\gotp=\cap_i M_{\gotP_i}.$$ 
Let $\frac{m}{s}$ be in the intersection, with $m\in M$ and
$s\in A$, $s\neq 0$. For every $i$ there exists an $m_i\in M$ and $s_i\in A\backslash \gotP_i$
such that $\frac{m}{s}=\frac{m_i}{s_i}$. Let $I$ be the ideal in $A$  generated by $s_1,s_2,\ldots$.
Then $I$ is not contained in the union of the $\gotP_i$'s,  by the prime avoidance lemma. So there is a $u=\sum_i x_i s_i\in I$
that is not in the union of the $\gotP_i$'s,  hence in none of them. Then
$$um=\sum_i x_is_im=s\sum_ix_im_i$$ 
and so 
$$\frac{m}{s}=\frac{\sum_ix_im_i}{u}.$$ 
Hence,
we can assume that $s$ is not in any of the $\gotP_i$'s, nor that any of the conjugates
of $s$ is in any of the $\gotP_i$'s. Using  the norm of $s$, 
we can assume that $s$ is in $B$ and not in any of the $\gotP_i$'s, hence not in $\gotp$.
So 
$$M_\gotp=\cap_i M_{\gotP_i}.$$
\end{proof}

\subsection{Modules of covariants are reflexive}
In this subsection, we shall reprove that all modules of covariants are finitely generated reflexive modules over the ring of invariants, using Samuel's Lemma~\ref{classic}. For another proof, see Brion~\cite{Brion}.

\begin{proposition} \label{reflexivecovariants}
Suppose $A$ is a normal graded algebra.
Every module of covariants $A^G(M)$ is a reflexive, finitely generated
graded $A^G$-module of rank equal to $\dim_kM$.
\end{proposition}

\begin{proof}
We first prove that $A^G(M)$ is a finitely generated, torsion free, graded $A^G$-module of rank $\dim_kM$.
Since $A^G(M)$ is a submodule of the torsion free $A^G$-module $A\otimes_kM$, it is torsion free.
Let $L$ be the quotient field of $A$. Then $K:=L^G$ is the quotient field of $A^G$, and
$$A^G(M)\otimes_{A^G}K\simeq (L\otimes_k M)^G.$$ 
By the normal basis theorem of Galois theory, there exists a $z\in L$ such that $\{\sigma(z);\sigma\in G\}$ is a $K$-basis of $L$. If $v_1,\ldots,v_m$ is a $k$-basis
for $M$, then it is easily shown that
$$\{ \sum_{\sigma\in G}(\sigma(z)\otimes \sigma(v_i)), \ 1\leq i\leq m\}$$ 
is a $K$-basis for 
$(L\otimes_k M)^G$. We conclude that  $A^G(M)$ has rank $\dim_k M$ over $A^G$.

Let 
$$S=\oplus_{i\geq 0} (A\otimes_k S^iM)=A\otimes_k \oplus_{i\geq 0} S^iM$$ 
be the symmetric algebra on the free $A$-module $A\otimes_k M$.
It is a finitely generated algebra, therefore its  invariant ring 
$$S^G=\oplus_{i\geq 0}A^G(S^iM)$$
is a finitely generated algebra over its subalgebra $A^G$. The homogeneous $A^G$-algebra generators
of $S^G$ that are contained in $A^G(M)$ are also homogeneous $A^G$-module generators of $A^G(M)$. So $A^G(M)$ is a finitely generated $A^G$-module.

By Lemma~\ref{classic} a finitely generated, torsion free graded
module is reflexive if every regular sequence $u,v$ on the ring is also a regular sequence on the module.  Let $u,v$ be a regular sequence on $A^G$. Since $A^G(M)$ is torsion free, $u$  acts regularly on $A^G(M)$. Let $\omega_1,\omega_2\in A^G(M)$ such that $v\omega_1=u\omega_2$.
Since $A$ is a reflexive $A^G$-module, cf.~\cite[VII \S 4.8 Corollaire]{Bourbaki},  $u,v$ is a regular sequence on $A\otimes_kM$. So, there is an $\omega\in A\otimes_kM$ such that $\omega_1=u\omega$. If $\omega$ is not in $A^G(M)$, then there exists a $\sigma\in G$ such that $\omega\neq \sigma(\omega)$. But   $\omega_1=\sigma(\omega_1)$
and $u=\sigma(u)$, so  $u(\omega-\sigma(\omega))=0$.  Since $u$  acts regularly on
$A\otimes_kM$, we get $\omega=\sigma(\omega)$, which is a contradiction. So $\omega\in A^G(M)$ and $u,v$ is a regular sequence on $A^G(M)$. We conclude that $A^G(M)$ is reflexive.
\end{proof}

\subsection{The Cohen-Macaulay property in the non-modular situation}
Since modules of covariants are reflexive, every regular sequence of length
two on $A^G$ is also a regular sequence on any non-zero module of covariants. One might ask
whether the same holds for longer regular sequences. In the non-modular situation  it is known that if $A$ is Cohen-Macaulay then $A^G$ is
also  Cohen-Macaulay, by Hochster-Eagon~\cite{HE}. More generally, in that context, all modules of covariants 
are Cohen-Macaulay.  In the non-modular situation it is also known that if $A$ is free over $A^G$, then all modules of covariants are free. Both results are no longer true in the modular situation, we give counter examples later on.

\begin{proposition}
Let $A$ be a graded algebra without zero-divisors, $G$ a finite group of automorphisms on $A$,  
$H<G$ a subgroup, and $M$ a finite dimensional $kG$-module. Suppose that $G/H$ is non-modular.

(i) Suppose $A^H(M)$ is a Cohen-Macaulay $A^H$-module. 
Then $A^G(M)$  is a Cohen-Macaulay $A^G$-module.
In particular, if  $A^H$ is Cohen-Macaulay, then 
$A^G$ is also Cohen-Macaulay. 
If, in addition, $H$ acts trivially on $M$, then 
$A^G(M)$ is a Cohen-Macaulay $A^G$-module.

(ii)
Suppose  $A^G\subset A^H$ is a free graded extension,
and $A^H(M)$ free over $A^H$. Then $A^G(M)$ is free over $A^G$.
In particular, if $H$ acts trivially on $M$, then $A^G(M)$ is free.
\end{proposition}

\begin{proof}
Since $G/H$ is non-modular, the operator 
$$\frac{1}{|G/H|}\Tr^G_H: (A\otimes_k M)^H\to (A\otimes_kM)^G:
\omega\mapsto \frac{1}{|G/H|}\sum_{gH\in G/H} g\omega$$
is $A^G$-linear and the identity when restricted to the submodule $(A\otimes_k M)^G$.
So $A^G(M)$ is an $A^G$-direct summand of $A^H(M)$. 

(i) Suppose $A^H(M)$ is a Cohen-Macaulay $A^H$-module. Let $f_1,\ldots,f_n$ be a homogeneous system of parameters for $A^G$. We have to show that $f_1,\ldots,f_n$ is a regular sequence
on $A^G(M)$. Let
$R=k[f_1,\ldots,f_n]$ be the polynomial algebra it generates. Then  $A^H(M)$ is also
a Cohen-Macaulay $R$-module. By the Auslander-Buchsbaum equation, it
follows that $A^H(M)$ is a projective graded $R$-module, hence free, by Nakayama's lemma for connected graded algebras, see~\cite{Eisenbud}.
Since a direct summand of a free graded
$R$-module is also a free graded $R$-module, it follows that $A^G(M)$ is a free graded
$R$-module. In particular, $f_1,\ldots,f_n$ forms a regular sequence on $A^G(M)$, and so 
$A^G(M)$ is a Cohen-Macaulay $A^G$-module.

(ii) If $A^H$ is free over $A^G$,  then the direct summand $A^G(M)$ of the free graded $A^G$-module $A^H(M)$ is also free.
\end{proof}

\subsection{Linear actions with large fixed point spaces}
In the linear situation, we can give some additional general results.
If $G$ acts linearly on $V$ with fixed points space $V^G$ of codimension one,
then $k[V]^G$ is a polynomial ring and all its modules of covariants are free. A similar result holds
if $\codim_VV^G=2$ and $k[V]^G$ is a polynomial algebra. The algebraic-geometric proof of the result depends in an essential way
on the linearity of the action and is based on an idea going back to Nakajima~\cite{Nakajima1983}. 

\begin{proposition}\label{codimensionone}
Suppose $G$ acts linearly on $V$, and let $M$ be a finite dimensional $kG$-module.

(i) Suppose $\codim_VV^G=1$. Then
$k[V]^G$ is a polynomial algebra and  $(k[V]\otimes M)^G$ is free as a graded module over $k[V]^G$.

(ii) Suppose $\codim_VV^G=2$. Then $(k[V]\otimes M)^G$ is a 
Cohen-Macaulay graded module over $k[V]^G$. In particular, $k[V]^G$ is a Cohen-Macaulay algebra.
Furthermore, if $k[V]^G$ is a polynomial algebra, then every module of covariants 
 is free.
 \end{proposition}

\begin{proof}
Let $R$ be a graded algebra with maximal ideal $R_+$ and consider $k=R/R_+$ as $R$-module.
Then a finitely generated graded $R$-module $M$ is free over $R$ {\em if and only if}
$\Tor^R_i(M,k)=0$, for $i\geq 1$. Let $k\subset K$ be a field extension, put
$R_K=R\otimes_kK$ and $M_K=M\otimes_kK$. Then 
$\Tor^R_i(M,k)\otimes_kK\simeq\Tor^{R_K}_i(M_K,K)$. So $M$ is free over $R$ {\em if and only if}
$M_K$ is free over $R_K$.
Let $f_1,\ldots,f_n$ be a homogeneous system of parameters of $k[V]^{G}$. Then it is also a 
homogeneous system of parameters for $K[V]^{G}$. Put $R=k[f_1,\ldots,f_n]$.
Then $k[V]^{G}(M)$ is a Cohen-Macaulay graded $k[V]^{G}$-module {\em if and only if}
it is free over $R$. Finally $k[V]^G$ is polynomial {\em if and only if} $k[V]$ is a free $k[V]^G$-module.
From these remarks it follows that we can, without loss of generality, suppose that $k$ is algebraically closed.  

The orbit space
$V/G$ is an affine algebraic variety with coordinate ring $k[V]^G$. The morphism of algebraic varieties  $V\to V/G$ associating
$v$ to its orbit $Gv$ corresponds to the inclusion $k[V]^G\subset k[V]$.
The linear algebraic group $U:=V^G$ acts on
$V$ by translation, commuting with the $G$-action, so $U$ also acts on the orbit space $V/G$ with orbits of dimension $\dim V^G$. 

(i) The singular locus of $V/G$ is closed, see~\cite[Cor. 16.20]{Eisenbud}, and $U$-stable. So, if it is non-empty,
its dimension is at least $\dim V^G=\dim V -1$. But, since $V/G$ is normal, its singular locus is
of codimension at least two, see~\cite[Th. 11.5]{Eisenbud}. This gives a contradiction. So $V/G$ is non-singular. Therefore $k[V]^G$ is
a regular graded algebra, hence a polynomial algebra.
Analogously, the non-free locus of $(k[V]\otimes M)^G$ is a closed set, defined by a suitable Fitting
ideal of a finite presentation, see~\cite[Prop. 20.8]{Eisenbud} or \cite[Theorem 4.10]{Matsumura}, and stable under the action of $U$.
So, if it is not empty, its codimension is one. But, since it is torsion free and every finitely generated torsion free module
over a discrete valuation ring is free, it follows that $(k[V]\otimes M)^G$ is free in codimension one.
A contradiction. So freeness follows.

(ii)
Under the hypothesis of (ii), the algebraic group $U$ acts on $V/G$  with orbits of codimension two.
Suppose $(k[V]\otimes M)^G$ is not Cohen-Macaulay. Then its non-Cohen-Macaulay locus in $V/G$ is non-empty, $U$-stable, closed (see~\cite[Cor. 6.11.3]{EGA}),
and contains $\pi(0)$; so contains the whole of $\pi(V^G)$.
Let $\gotP$ be the linear ideal defining $V^G$; it is a prime ideal of height $2$. We conclude that $(k[V]\otimes M)^G$ is not Cohen-Macaulay at $\gotp:=\gotP\cap k[V]^G$.

Write $B$ for the localisation of $k[V]^G$ at $\gotp$ and $N$ for the
localisation of $k[V]^G(M)$ at $\gotp$. Then $B$ is a normal local ring of dimension two, hence
Cohen-Macaulay and $N$ is a reflexive $B$-module. And so every regular sequence
(of length two) is also a regular sequence on $N$, by Lemma~\ref{classic} (this lemma remains
true in the case where $B$ is only a noetherian integrally closed domain). We conclude that
$N$ is Cohen-Macaulay. But this is a contradiction; so $(k[V]\otimes M)^G$ is Cohen-Macaulay.
We conclude by remarking that every graded Cohen-Macaulay module for a polynomial algebra is free, by Auslander-Buchsbaum's formula, cf.~\cite[Thm. 19.9]{Eisenbud}.
\end{proof}

\begin{remark}
The freeness in (i) was proved differently
by Hartmann and Shepler~\cite{HartmannS}.
\end{remark}

A similar proof shows that the freeness (or Cohen-Macaulay) property   of modules of covariants for $G$ descends to
modules of covariants of point-stabilizers.

\begin{proposition}\label{pointstabilizer}
Suppose $G$ acts linearly on $V$. Let $U\subset V$ be a linear subspace with point-stabilizer $G_U$, and
let $M$ be a finitely generated $kG$-module.

(i) If $k[V]^G(M)$ is free over $k[V]^G$, then
$k[V]^{G_U}(M)$ is free over $k[V]^{G_U}$.

(ii) If $k[V]^G(M)$ is Cohen-Macaulay over $k[V]^G$, then
$k[V]^{G_U}(M)$ is Cohen-Macaulay over $k[V]^{G_U}$.
\end{proposition}

\begin{proof}
As in the proof of Proposition~\ref{codimensionone}, we can assume that $k$ is algebraically closed.
The linear  algebraic group $U$ acts on $V$
by translations, commuting with the $G_U$-action. Let $\gotP$ be the prime ideal generated
by the linear forms vanishing on $U$, and put
$\gotp=\gotP\cap k[V]^{G_U}.$
Let $v\in U$ such that $G_v=G_U$ (we can find such a $v$ since $k$ is algebraically closed),
and let $\gotM_v\supset \gotP$ be the corresponding maximal ideal.  Then $G_U$ coincides with the decomposition subgroup of $\gotM_v$. Put 
$\gotm_v:=\gotM_v\cap k[V]^{G_U};$ 
it is a maximal ideal containing $\gotp$. There exists
a translation in $U$ that moves the maximal graded ideal of $k[V]^{G_U}$ onto $\gotm_v$.

Suppose $k[V]^{G_U}(M)$ is not free (or Cohen-Macaulay), then it is not free (or Cohen-Macaulay)
at the maximal graded ideal and (using the commuting translation action) 
therefore not free (or Cohen-Macaulay) at the maximal ideal $\gotm_v$ either. This gives a contradiction
with Lemma~\ref{hulp} in the last section.
\end{proof}

\section{Generalizations of results of Stanley and Nakajima}
\subsection{Free modules of covariants}
Stanley~\cite{Stanley} and Nakajima~\cite{Na} already extensively studied modules of
semi-invariants. One of their fundamental results can be described as follows.
Let $\lambda:G\to k^\times$ be a linear character with module of semi-invariants $A^G_\lambda$.
Then if $A^G_\lambda$ is free,  $A^W_\lambda$ is also free
and both modules of semi-invariants share the generator. 
For modules of covariants, this is generalized as follows.

\begin{theorem}\label{tensor}
Let $A$ be a normal graded algebra,  $G$ a finite group of automorphisms of
$A$, and $M$ a finite dimensional 
$kG$-module. Let
$K$ be a subgroup such that $W\leq K\leq G$.
Suppose that $A^G(M)$ is free over $A^G$. Then multiplication induces an isomorphism of graded $A^K$-modules
$$\mu:A^K\otimes_{A^G}A^G(M)\simeq A^K(M).$$
In particular, $A^K(M)$ is free over $A^K$, and $A^G(M)$ and $A^K(M)$ share bases.
\end{theorem}

\begin{proof}
By assumption, $A^K\otimes_{A^G}A^G(M)$ is free over $A^K$ of rank $\dim_kM$. So both sides are reflexive of the same rank. And since 
$$s_{A^K}(A^K\otimes_{A^G}A^G(M))=s_{A^G}(A^G(M))=s_{A^K}(A^K(M)),$$ by Theorem~\ref{BCB}, both sides also have the same $s$-invariant. Let $L$ be the quotient field of $A$.
To prove that $\mu$ generically is an isomorphism, we must show that the $L^K$-linear map
\begin{equation}\label{generic}
L^K\otimes_{L^G}(L\otimes_k M)^G\to (L\otimes_k M)^K
\end{equation}
is an isomorphism of $L^K$-vector spaces of dimension $m=\dim_kM$. It would follow that $\mu$ is a
pseudo-isomorphism, and then, by Lemma~\ref{reflexive}(i), that $\mu$ is an isomorphism.
To show that (\ref{generic}) is an isomorphism, it suffices to show  injectivity, or that
any $L^G$-basis of $(L\otimes_k M)^G$ is also an $L^K$-basis of $(L\otimes_k M)^K$.
Let $\omega_1,\ldots,\omega_m$ be an $L^G$-basis of $(L\otimes_k M)^G$.
Suppose we have an $L^K$-linear relation $\sum_i u_i \omega_i=0$, where $u_i\in L^K$.
Since the $\omega_i$'s are $G$-invariant, it follows for all $v\in L^K$ that 
$$\sum_i \Tr^G_K(v u_i) \omega_i=0,\
\hbox{ where}\ 
\Tr^G_K:=\sum_{\sigma K\in G/K}\sigma:L^K\to L^G.$$
Since the $\omega_i$ are independent over $L^G$, it follows that for all
$v\in L^K$ and $i$ we have 
$$\Tr^G_K(vu_i)=0.$$ 
But since the field extension $L^G\subset L^K$ is separable, the
map $\Tr^G_K:L^K\to L^G$ is non-zero, by~\cite[VI Theorem 5.2]{Lang}. So each $u_i$ is zero, and the $\omega_1,\ldots,\omega_m$ are also independent
over $L^K$, hence forms a basis. This finishes the proof.
\end{proof}

\subsection{Jacobian criterion for freeness}
Let $G$ act linearly on the vector space $V$, and let $\lambda:G\to k^\times$ be a linear character. Stanley~\cite{Stanley} and Nakajima~\cite{Na} show that $k[V]^W_\lambda$ is always free of rank one over $k[V]^W$, and they construct
a generator $f_\lambda$ of $k[V]^W_\lambda$. It is a product of linear forms; let $e_\lambda$
be  its degree. They prove that $k[V]^G_\lambda$ is free over $k[V]^G$ {\em if and only if} $k[V]^G_\lambda$ contains
a non-zero element of degree $e_\lambda$; in that case $f_\mu$ is the generator. 

This freeness criterion was generalized,
in some sense, to modules of covariants in characteristic zero  by Gutkin ~\cite{Gutkin}, \cite{OS}.
Let $M$ be an $m$-dimensional $kG$-module with a fixed basis $v_1,\ldots,v_m$.
For $m$ homogeneous elements
$$\omega_j=\sum_{i=1}^m a_{ij}\otimes v_i\in k[V]\otimes_k M,\ 1\leq j\leq m,$$
  define the {\em Jacobian determinant} by
$$\Jac_M(\omega_1,\ldots,\omega_m):=\det \left(a_{ij}\right)_{1\leq i,j\leq m}\in k[V].$$
We formulate the Jacobian criterion of freeness.  $k[V]^G(M)$ is free over $k[V]^G$ {\em if and only if} there is an $m$-tuple  $\omega_1,\ldots,\omega_m$ of homogeneous elements
in $k[V]^G(M)$ whose Jacobian determinant $\Jac_M(\omega_1,\ldots,\omega_m)$ is non-zero and of degree $e_M$,
where $e_M$
is the $s$-invariant $s_{k[V]^G}(k[V]^G(M))$.

Next, we associate to $M$ a certain product of linear forms $F_M$.
Let $U\subset V$ be a linear hyperplane  defined by the linear form $x_U$ and
let $G_U$ be its point-stabilizer. The module of covariants $k[V]^{G_U}(M)$ is
free, by Proposition~\ref{codimensionone}, with homogeneous generators $w_1,\ldots,w_m$, say. Put
$e_U(M)$ for the sum of their degrees. We claim that
$\Jac_M(w_1,\ldots,w_m)$ is equal to $x_U^{e_U(M)}$, up to a non-zero scalar in $k^\times$.
There are only finitely many $U$'s such that $G_U\neq 1$, so the following product of linear forms
is well-defined
$$F_M:=\prod_{U\subset V,\ \codim_VU=1}x_U^{e_U(M)}.$$
Its degree turns out to be
$e_M=\sum_U e_U(M)$.

Write $\lambda:G\to k^\times$ for the linear character associated to the $kG$-module $(\wedge^mM)^*$.
If $\omega_1,\ldots,\omega_m$ is an $m$-tuple of homogeneous elements of $k[V]^G(M)$, then
the Jacobian determinant $\Jac_M(\omega_1,\ldots,\omega_m)$ is a 
$\lambda$-semi-invariant.
Put $J_M^G$ for the $k[V]^G$-submodule of $k[V]^G_\lambda$ spanned by all such 
Jacobian determinants.
We'll show that any element of $J_M^G$ is divisible by $F_M$.

The following generalizes a result due to Gutkin in characteristic zero~\cite{Gutkin}, \cite{OS}.

\begin{theorem}\label{Jacobian}
Let $G$ act linearly on $V$, and let $M$ be a finite dimensional $kG$-module.

(i) For any linear subspace $U\subset V$ of codimension one we have
$$J_M^{G_U}=k[V]^{G_U}\cdot x_U^{e_U(M)}\subseteq k[V]^{G_U}_{\lambda_U},$$
where $\lambda_U=\lambda|_{G_U}$ and $e_U(M)$ equals the $s$-invariant
$s_{k[V]^{G_U}}(k[V]^{G_U}(M))$.

(ii) $F_M$ is a greatest common divisor of all the elements in $J_M^G$ and $$J_M^G\subseteq k[V]\cdot F_M.$$

(iii) {\em [Jacobian criterion]} Let $\omega_1,\ldots,\omega_m$ be an $m$-tuple of homogeneous elements of $k[V]^G(M)$. The following three
statements are equivalent.

(a) The module of covariants $k[V]^G(M)$ is free over $k[V]^G$ with basis $\omega_1,\ldots,\omega_m$;

(b) There is a non-zero scalar $c\in k^\times$ such that $$\Jac_M(\omega_1,\ldots,\omega_m)=cF_M;$$

(c) $\sum_{i=1}^m\deg \omega_i=s_{k[V]^G}(k[V]^G(M))$ and
$\Jac_M(\omega_1,\ldots,\omega_m)\neq 0.$
\end{theorem}

\begin{proof}
The proof is given in subsection~\ref{proofjacobian}.
\end{proof}

We get a quick proof of  a part of Theorem~\ref{tensor} in the linear situation, together with its converse. 
\begin{corollary}\label{converse}
Let $G$ act linearly on $V$, $M$ be a finite dimensional $kG$-module,
and $K$ be a subgroup such that $W\leq K\leq G$.

Then  $k[V]^G(M)$ is free over $k[V]^G$  if and only if
$k[V]^K(M)$ is free over $k[V]^K$ and $k[V]^G(M)$ and $k[V]^K(M)$ share generators.
\end{corollary}

\begin{proof}
Suppose $k[V]^G(M)$ is free over $k[V]^G$ with homogeneous basis $\omega_1,\ldots,\omega_m$. Then we get, by the Jacobian
criterion,  that
$\Jac_M(\omega_1,\ldots,\omega_m)$ is non-zero of degree $s_{k[V]^G}(k[V]^G(M))$.  These basis elements
are also elements of $k[V]^K(M)$ and we have the same $s$-invariants
$$s_{k[V]^G}(k[V]^G(M))=s_{k[V]^K}(k[V]^K(M)),$$ 
by Theorem~\ref{BCB}. By the Jacobian criterion, $k[V]^K(M)$ is free with homogeneous basis $\omega_1,\ldots,\omega_m$ and so $k[V]^K(M)$ and $k[V]^G(M)$ share generators.

Conversely, suppose $A^K(M)$ is free over $A^K$ and $A^G(M)$ and $A^K(M)$ share generators.
This means that $A^K(M)$ can be generated by $G$-invariant elements, and hence we can extract a
homogeneous basis $\omega_1,\ldots,\omega_m$ of $A^K(M)$ consisting of $G$-invariant elements.
From the Jacobian criterion it follows that 
$\Jac_M(\omega_1,\ldots,\omega_m)$ is non-zero of degree 
$$s_{k[V]^K}(k[V]^K(M))=s_{k[V]^G}(k[V]^G(M)).$$ 
So, by the Jacobian criterion once again, it follows
that the $\omega_1,\ldots,\omega_m$ form a basis of the free $k[V]^G$-module $k[V]^G(M)$.
\end{proof}

\begin{example} We continue example~\ref{firstexample}.
Let us prove that $k[V]^{H_2}(M_1)$ indeed has basis $1\otimes x_1$, $1\otimes x_2$, and
$x_2\otimes x_3+x_3\otimes x_2$, like we claimed before. First of all, the Jacobian determinant of this triple is
$x_2$, so $s(k[V]^{H_2}(M_1))\leq 1$. But $s(k[V]^{H_2}(M_1))=0$ is impossible, since
$H_2$ does not act trivially on $M_1$. So $s(k[V]^{H_2}(M_1))=1$, and by the Jacobian criterion
the three given covariants form a basis.

Next, we consider the three $G$-covariants of type $M_1$: $1\otimes x_1$, $1\otimes x_2$, and
$x_4(x_1+x_4)\otimes x_1+x_3(x_1+x_3)\otimes x_2+x_2(x_1+x_2)\otimes x_3$. Their Jacobian determinant is $x_2(x_1+x_2)$,  hence non-zero of degree $2=s_{k[V]^G}(k[V]^G(M_1))$. So, by
the Jacobian criterion,  they form a basis of the free module of covariants $k[V]^G(M_1)$ and
$F_{M_1}=x_2(x_1+x_2)$.

Similarly, $1\otimes x_1$, $1\otimes x_2$, $x_4\otimes x_1+x_3\otimes x_2+x_2\otimes x_3+x_1\otimes x_4$, and $x_4^2\otimes x_1+x_3^2\otimes x_2+x_2^2\otimes x_3+x_1^2\otimes x_4$ 
is a basis of the free module of covariants $k[V]^G(M_2)$, with
$F_{M_2}=x_1x_2(x_1+x_2)$.

We remark that $k[V]^G$ is not a polynomial ring, but is minimally generated by
$x_1$, $x_2$, $x_1x_4(x_1+x_4)+x_2x_3(x_2+x_3)$, $x_3^4+(x_1^2+x_1x_2+x_2^2)x_3^2+x_1x_2(x_1+x_2)x_3$
and $x_4^4+(x_1^2+x_1x_2+x_2^2)x_4^2+x_1x_2(x_1+x_2)x_4$.
\end{example}

\begin{example}
Let $G$ act linearly on the vector space $V$ of dimension $n$.
Fix a basis $x_1,\ldots,x_n$ of linear forms. For any invariant $f\in k[V]^G$, write $df$ for the covariant
$$df:=\sum_i\frac{\partial f}{\partial x_i}\otimes x_i\in (k[V]\otimes_kV^*)^G=k[V]^G(V^*).$$
Let  $f_1$, $f_2, \ldots,f_n$ be an $n$-tuple of $G$-invariants, then 
$$\Jac_{V^*}(df_1,\ldots,df_n)=\det\left(\frac{\partial f_i}{\partial x_j}\right)_{1\leq i,j \leq n}.$$
Write $e_{V^*}:=s_{k[V]^G}(k[V]^G(V^*))$. Then our Jacobian criterion says that $df_1$, $df_2$,
\ldots, $df_n$ (freely) generate $k[V]^G(V^*)$ if and only if 
$\sum_i \deg df_i=\sum_i(\deg f_i-1)=e_{V^*}$
and
$$\Jac_{V^*}(df_1,\ldots,df_n)\neq 0.$$

On the other hand the classical Jacobian criterion for $k[V]^G$ to be polynomial is as follows, 
see~\cite[Criterion 2]{Br7}.  $k[V]^G=k[f_1,\ldots,f_n]$  {\em if and only if }
$ \sum_i(\deg f_i-1)=\delta$
and  $\det\left(\frac{\partial f_i}{\partial x_j}\right)_{1\leq i,j \leq n}\neq 0$, where
$\delta$ is the differential degree.

We conclude that  if $\delta=e_{V^*}$,  then $k[V]^G=k[f_1,\ldots,f_n]$ (hence is a polynomial algebra) {\em if and only if}
$df_1$, $df_2$,\ldots, $df_n$ (freely) generate $k[V]^G(V^*)$.

By example~\ref{secondexample}, if the base field equals $\F_2$, we have the equality $\delta=e_{V^*}$. 

We claim that if $G$ has no transvections, $\delta=e_{V^*}$ also holds.
Write $\det$ for the linear character associated to the $kG$-module $\wedge^nV$. Let $U\subset V$
be a linear subspace of codimension one, defined by the linear form $x_U$. The point-stabilizer $G_U$ is cyclic, say with generator $\sigma$ of order $h_U$, since there are no transvections.
There is an eigenbasis of linear forms $x_1=x_U$, $x_2,\ldots,x_n$; the eigenvalue of $x_1$ is
a primitive $h_U$-th root of unity in $k$, the other $x_i$'s have eigenvalue one. Then $$k[V]^{G_U}=k[x_1^{h_U},x_2,\ldots,x_n]$$ 
and $k[V]^G(V^*)$ is freely generated by
$x_1^{h_U-1}\otimes x_1,1\otimes x_2,\ldots,1\otimes x_n$
with Jacobian determinant
$x_1^{h_U-1}$. But $x_1^{h_U-1}$ is also the generator of  the module of $\det$-semi-invariants.
We conclude  that $F_{V^*}$ equals the generator of $W$-semi-invariants of type $\det$, and so 
its degree $e_{V^*}$ equals $\delta$.
\end{example}

\begin{remark}
It is conjectured that if $k[V]^G$ is polynomial, at least the module of covariants
$(k[V]\otimes V^*)^G$ is free, see~\cite{HartmannS}. If there are no transvections in $G$,  this is true, by the example above, and was known before 
by work
of Knighten~\cite{Knighten} or Hartmann~\cite{Hartmann}.  Knighten proved that if there are no transvections, then
$k[V]^G(V^*)$ is isomorphic to $\Omega_{k[V]^G/k}^{**}$, the reflexive closure of
the module of differentials. If $k[V]^G$ is a polynomial algebra, then the module
of differentials $\Omega_{k[V]^G/k}$ is free, and hence isomorphic to the module
of covariants $k[V]^G(V^*)$; which is therefore free. The modules of covariants $k[V]^G(\wedge^i(V^*))$ are then also free;  the arguments  in characteristic zero given by  Shepler~\cite{Shepler}  carry over to the general  transvection free case.
\end{remark}

\section{Free extensions of algebras of invariants}\label{subgroupinvariants}
Let $H<G$ be a subgroup and let $G$ act as before on the graded algebra $A$. It is of interest
to study the ring of invariants $A^H$ as a module over its subring $A^G$. For example,
if $A^H$ is free over $A^G$,  both rings have a very close ring theoretic relationship, e.g.,
one of the two is Cohen-Macaulay {\em if and only if} the other is.
Fortunately,  $A^H$ is as an $A^G$-module isomorphic to a  module of covariants, in particular
even $A$ itself, so we can
apply the theorems we have obtained before.
We start by showing somewhat more generally that modules of covariants of subgroups are modules of covariants.

\begin{lemma}\label{alsocovariants}
Let $A$ be a graded algebra, $G$ a finite group of automorphisms of $A$, $H<G$ a subgroup, and
$M$ a finite dimensional $kH$-module.  There is
an isomorphism of $A^G$-modules:
$$(A\otimes_k M)^H\simeq (A\otimes_k\Ind_H^GM)^G.$$
In particular, the ring of invariants $A^H$ is as  $A^G$-module isomorphic to a module of covariants
$$A^H\simeq A^G(\Ind_H^G k).$$
Here we take $\Ind_H^GM$ to be $kG\otimes_{kH}M$, and in particular $\Ind_H^Gk$ is the
permutation $kG$-module on the left coset space $G/H$.
\end{lemma}

\begin{proof}
This can be proved using Frobenius reciprocity of representation theory, or directly as follows.
Fix left coset representatives
$g_1,\ldots,g_s$ of $G/H$, where $g_1=1$, and a basis $v_1,\ldots,v_m$ of $M$. 
We will show that the $A^G$-linear map 
$$\psi:A^H(M)\to A^G(\Ind_H^GM): \sum_j a_j\otimes v_j\mapsto \Tr_H^G(\sum_j a_j\otimes g_1\otimes v_j),$$
is an isomorphism. Here $\Tr_H^G=\sum_i g_i$, $a_j\in A$ and $g_1=1$ is considered as
the unit in $kG$. 
Let $\omega\in (A\otimes_k\Ind_H^GM)^G$ be any element. It can be uniquely written as
$\omega=\sum_{i,j} a_{ij}\otimes g_i \otimes v_{j},$
where of course $a_{ij}\in A$, and $g_i\in G$ is seen as a basis element of $kG$ over $kH$.
We can write $\omega=\sum_i \omega_i,$
where $\omega_i=\sum_{j} a_{ij}\otimes g_i \otimes v_{j}.$
Since $\omega$ is $G$-invariant, we get $g_i\omega=\omega$, and $g_i\omega_1=\omega_i$
(since $g_1=1$ and by the unicity of the expression). So 
$$\omega=\sum_ig_i\omega_1=\Tr^G_H\omega_1.$$ It also follows that $\omega_1$ is $H$-invariant, and so
$\sum_{j} a_{1j}\otimes v_{j}\in A^H(M).$
Therefore, $\psi$ is surjective.
We remark that $\omega=0$ {\em if and only if} $\omega_1=0$, hence injectivity. 
\end{proof}

\begin{remark}
Suppose $G$ acts linearly on $V$ such that  the invariant ring $k[V]^G$ is polynomial.
If $G$ is non-modular then all modules of covariants are free. It is natural to ask, whether
this remains true when $G$ is modular. The answer is no.

There are many known examples of groups $H$ acting linearly on a vector space $V$ over
a finite field $\F_q$ whose ring of invariants $\F_q[V]^H$  is not Cohen-Macaulay,
see~\cite{DerksenKemper}.
Let $G=\GL(V)$ be the full linear group. Then $\F_q[V]^G$ is a polynomial ring (generated by
the so-called Dickson-invariants, see \cite{Benson1993}). Since $\F_q[V]^H$ is not Cohen-Macaulay
it is not free over the polynomial subring $\F_q[V]^G$.
But $\F_q[V]^H$ is a special module of covariants for $\F_q[V]^G$. We conclude that some
modules of covariants for $\F_q[V]^G$ are not free, even though $\F_q[V]^G$ is polynomial.
\end{remark}

\subsection{Generalization of Serre's theorem}
For a linear action of $G$ on a vector space $V$ the invariant  algebra $k[V]^G$ is a polynomial algebra
{\em if and only if} the extension $k[V]^G\subset k[V]$ is free. By a theorem of Serre the group is  then generated by reflections, cf.~\cite[p.~85-86]{Benson1993}. The converse is true when $G$ is non-modular.
We  give a generalization of Serre's result.

\begin{theorem}\label{Serre}
Let $A$ be a normal graded algebra with a finite group $G$ acting on it by graded algebra automorphisms.
Let $H$ and $K$ be subgroups of $G$ such that $K$ contains the reflection subgroup $W$ of $G$
acting on $A$. Suppose $A^G\subset A^H$ is a free graded extension.

(i) The group $G$ is generated by $H$ and the reflections on $A$ contained in $G$, i.e., $G=WH$.

(ii) Multiplication induces an isomorphism
$$A^K\otimes_{A^G}A^H\simeq A^{K\cap H}.$$
In particular, $A^K\subset A^{K\cap H}$ is also a free graded extension, and 
$$A^KA^H=A^{K\cap H},$$ 
i.e., the ring of 
$K\cap H$-invariants is generated by the $K$-invariants and the $H$-invariants. 

(iii) Furthermore,
$A^G\subset A^H$ is a graded Gorenstein extension (or complete intersection extension)
{\em if and only if} $A^K\subset A^{K\cap H}$ is a graded Gorenstein extension (or complete intersection extension).
\end{theorem}

\begin{proof}
(i) By Lemma~\ref{alsocovariants} we can consider $A^H$ as a module of covariants over $A^G$ by
$$A^H\simeq A^G(\Ind_H^G k).$$
Since $W\leq WH\leq G$ and $A^H$ is free over $A^G$, we can apply Theorem~\ref{tensor}
and obtain an isomorphism of $A^{WH}$-modules
$$\mu:A^{WH}\otimes_{A^G}A^G(\Ind_H^G k)\simeq A^{WH}(\Ind_H^G k)$$
and so $A^{WH}(\Ind_H^G k)$ is a free graded $A^{WH}$-module.
We remark that $\Ind_H^G k$ is a permutation $kWH$-module containing the permutation
submodule $\Ind_H^{WH}k$, which is therefore a $kWH$-direct summand. It follows
that $A^{WH}(\Ind_H^{WH}k)$ is a direct summand of $A^{WH}(\Ind_H^Gk)$ as graded
$A^{WH}$-modules, and it is therefore free as well, since finitely generated graded projective modules are free by Nakayama's lemma for connected graded algebras. Hence
$$A^H\simeq A^{WH}(\Ind_H^{WH}k)$$ 
is a free graded $A^{WH}$-module. In particular, $A^{WH}$ is a direct summand of $A^H$
as graded  $A^{WH}$-module, and therefore also as graded $A^G$-module. Since $A^H$ is
free over $A^G$, it follows again that $A^{WH}$ is also free over $A^G$.

In the remaining part of the proof of (i), we shall use the basic properties of the Noether different $\Diff^N$ (also called homological different) and the Dedekind
different $\Diff^D$ of the extension, cf.~\cite{Benson1993} or \cite{Br6}.
By a theorem of Noether and Auslander-Buchsbaum, see \cite[Theorem 3.11.1]{Benson1993} or
\cite[Lemma 2(iii)]{Br6},  freeness implies
that 
$$\Diff^N_{A^{WH}/A^G}=\Diff^D_{A^{WH}/A^G}.$$ 
Since $\Diff^D_{A^{WH}/A^G}=(1)$, we get $\Diff^N_{A^{WH}/A^G}=(1)$. 
This means that $1$ annihilates the kernel of the multiplication map 
$$A^{WH}\otimes_{A^G}A^{WH}\to A^{WH}.$$ 
Hence it is an isomorphism.
Now $A^{WH}$ is free over $A^G$, say of rank $r$. Therefore $A^{WH}\simeq A^{WH}\otimes_{A^G}A^{WH}$
has rank $r$ over $A^{WH}$, hence $r=1$. We conclude that $A^G=A^{WH}$ and so from Galois theory it follows that indeed $G=WH$.

(ii) Since $G=WH=KH$, it follows that $\Ind_H^Gk=\Ind_{K\cap H}^Kk$ as $kK$-modules.
Now we can apply Theorem~\ref{tensor} and using that 
$$A^H\simeq A^G(\Ind_H^Gk)\ {\rm and}\ A^{K\cap H}\simeq A^K(\Ind_{K\cap H}^Kk)$$ 
we obtain the isomorphism
$$\mu:A^K\otimes_{A^G}A^H\simeq A^{K\cap H}.$$ From this the remaining assertions in (ii) and
(iii) follow easily.
\end{proof}

In the linear case, we also have a converse.
\begin{proposition}\label{Serreconverse}
Let the finite group $G$ act linearly on $V$ and let $H$ and $K$ be subgroups such that $W\leq K\leq G$.
Suppose $G=KH$, $k[V]^{K\cap H}$ is free over $k[V]^K$
and $k[V]^{K\cap H}$ is generated by $k[V]^K$ and $k[V]^H$. 
Then $k[V]^H$ is free over $k[V]^G$.
\end{proposition}

\begin{proof}
Since $k[V]^{K\cap H}$ is generated by $k[V]^K$ and $k[V]^H$, the $k[V]^K$-module
homomorphism 
$$\mu: k[V]^K\otimes_{k[V]^G}k[V]^H\to k[V]^{K\cap H}$$
induced by multiplication is surjective. So the $k[V]^G$-module $k[V]^H$ and the
$k[V]^K$-module $k[V]^{K\cap H}$ share generators. Since $G=KH$, we get that
$\Ind_H^Gk$ and $\Ind^K_{K\cap H}k$ are isomorphic as $kK$-modules. Hence
$$k[V]^{K\cap H}\simeq k[V]^K(\Ind^K_{K\cap H}k)\simeq  k[V]^K(\Ind_H^Gk)$$ 
as
$k[V]^K$-modules. And 
$$k[V]^G(\Ind_H^Gk)\simeq k[V]^H$$ 
as $k[V]^G$-modules.
So $k[V]^K(\Ind_H^Gk)$ is free and shares generators with $k[V]^G(\Ind_H^Gk)$.
By Corollary~\ref{converse}, we conclude that $k[V]^G(\Ind_H^Gk)=k[V]^H$ is free
over $k[V]^G$.
\end{proof}

\subsection{Examples of free extensions}
In the linear case, Chevalley-Shephard-Todd's classical theorem says that if $G$ is generated by
reflections and
$G$ is non-modular, then the extension $k[V]^G\subset k[V]$ is a graded complete
intersection extension, cf.~\cite[Theorem 7.2.1]{Benson1993}, \cite[Theorem 6]{Br6}. 
This was generalized by Hochster-Eagon~\cite{HE}  to more general actions by
adapting  Chevalley's conceptual proof. An automorphism $\sigma$ of an integral domain $R$ 
is called a {\em Hochster-Eagon reflection} if there is a non-zero $f\in R$ such that
$\sigma(a)-a\in (f)$ for all $a\in R$. Hochster-Eagon reflections are also 
reflections in our sense, but not necessarily conversely. If $R$ is factorial, then reflections  are the
same as Hochster-Eagon reflections. 

\begin{proposition}[Hochster-Eagon, Avramov]
Let $A$ be a normal graded algebra and  $G$ a finite group of graded algebra automorphisms of $A$.
Suppose $H\normal G$, $G/H$ is non-modular and the action of $G/H$ on
$A^H$ is generated by Hochster-Eagon reflections.
Then $A^G\subset A^H$ is a graded complete intersection extension.
\end{proposition}

\begin{proof}
Hochster-Eagon~\cite{HE} proved the freeness of $A^G\subset A^H$, and Avramov~\cite{Av} noticed that the fiber algebra is a graded complete intersection. So, in other words, $A^G\subset A^H$
is a graded complete intersection extension.
\end{proof}

\begin{example}
An example of a graded complete intersection extension of a different kind is due to Nakajima~\cite{Na}.
Recall the definition of $G^1$ in subsection~\ref{notation}.

\begin{proposition}[Nakajima]
(i) Suppose that $A^G$ is factorial and $G^1\leq H\leq G$. Then $A^G\subset A^{H}$ is a free extension and $G/H$ is non-modular.

(ii) Suppose that $A$ is  factorial and $G$ is generated by reflections on $A$. Then 
$A^G\subset A^{G^1}$ is a graded complete intersection extension.
\end{proposition}

\begin{proof}
These results were (implicitly) proved by Nakajima~\cite{Na}.
For (i), we note that $A^{H}$ is isomorphic to the direct sum of all modules of semi-invariants of
types $\lambda$ having $H$ in the kernel.
Any module of semi-invariants is reflexive of rank one, hence isomorphic to a divisorial ideal. Since $A^G$ is factorial
by assumption, all divisorial ideals are principal ideals. So all modules of
semi-invariants are free and so $k[V]^G\subset k[V]^{H}$ is free. For (ii), we refer to 
Nakajima~\cite{Na}.
\end{proof}
\end{example}

\begin{example} 
An example of a modular free extension. 
Suppose $A^H$ is factorial, $HW=G$ and $|G/H|=p$ (the characteristic of $k$). Then $A^G$ is also
factorial, and $A^G\subset A^H$ is a graded complete intersection extension {\em if and only if }$A^G$ is
a direct summand of $A^H$, cf.~\cite{Br7}. The direct summand property holds, for example,
when $A^G\subset A$ is a free graded extension.
\end{example}

\section{A ramification formula}\label{ramification}
Let $A$ be a normal graded algebra with a finite group $G$ of $k$-algebra automorphisms, $M$ a finite dimensional $kG$-module, and $H<G$ a subgroup.
We use the transfer map
$$\Tr^H: A\otimes_k M\to (A\otimes_k M)^H: u\mapsto \sum_{\sigma \in H}\sigma(u)$$
and the relative transfer map
$$\Tr^G_H: (A\otimes_k M)^H\to (A\otimes_k M)^G: u\mapsto \sum_{\sigma H \in G/H}\sigma(u)$$
to construct the following homomorphisms:
$$\Phi^H:A\otimes_k M\to\Hom_{A^H}(A,(A\otimes_k M)^H):\ [\Phi^H(\omega)](a):=\Tr^H(a\omega)$$
where $\omega\in A\otimes_k M$, $a\in A$ and 
$$\Phi^G_H:\Hom_{A^H}(A,(A\otimes_k M)^H)\to \Hom_{A^G}(A,(A\otimes_k M)^G):\
\phi\mapsto \Tr^G_H\circ \phi.$$
Both $\Phi^H$ and $\Phi^G_H$ are injective $A$-module homomorphisms between reflexive $A$-modules and $$\Phi^G=\Phi_H^G\circ\Phi^H.$$

The following theorem is the technical heart of this article; all our main results rely on it.
\begin{theorem}\label{main}
Let $A$ be a normal graded algebra with a finite group $G$ of $k$-algebra automorphisms, $M$ a finite dimensional $kG$-module, and $H<G$ a subgroup.

(i) The $A$-module homomorphism of reflexive $A$-modules
$$\Phi^G_H:\Hom_{A^H}(A,(A\otimes_k M)^H)\to \Hom_{A^G}(A,(A\otimes_k M)^G)$$
is injective, generically an isomorphism, and an isomorphism at the prime ideal of height one $\gotP\subset A$ if we have an equality of inertia subgroups $G_i(\gotP)=H_i(\gotP)$. 

(ii) For any subgroup $K<G$ containing all reflections in $G$ on $A$, the map
$\Phi_K^G$ is an isomorphism.

(iii) Write $C^G_H$ for the cokernel of $\Phi^G_H$. Let $\gotP$ be a height one prime ideal of $A$ with inertia subgroups $H_i:=H_i(\gotP)$ and $G_i:=G_i(\gotP)$.   We have
$\left(C^G_H\right)_\gotP\simeq
\left(C^{G_i}_{H_i}\right)_\gotP,$ and so in particular,
$$\length_{A_\gotP}\left(\left(C^G_H\right)_\gotP\right)=
\length_{A_\gotP}\left( \left(C^{G_i}_{H_i}\right)_\gotP\right).$$
\end{theorem}

\begin{proof}
Fix  a prime ideal $\gotP\subset A$ with decomposition group 
$G_d$ and inertia subgroup $G_i$. Put 
$$\gotp:=\gotP\cap A^G,\ \gotP_i:=A^{G_i}\cap\gotP,\ \gotP_d:=A^{G_d}\cap \gotP,\ 
R:=(A^{G_d})_{\gotP_d},\ S:=(A^{G_i})_{\gotP_i}$$
and $\Gamma:=G_d/G_i$. It is known that $\gotP$ is the only prime ideal of $A$ lying above
$\gotP_d$ (or $\gotP_i$) (see e.g.~\cite[Satz 20.4]{SS}). For any $A$-module $N$, we get 
$N_{\gotP_d}\simeq N_{\gotP_i}\simeq N_{\gotP}$, and  $(A^{G_i})_{\gotP_d}\simeq (A^{G_i})_{\gotP_i}=S$.

If $B$ is a flat $A^G$-algebra with trivial $G$-action, then
$(A\otimes_{A^G}B)^G\simeq A^G\otimes_{A^G}B\simeq B$, and similarly,
$((A\otimes_{A^G}B)\otimes_k M)^G\simeq (A\otimes_kM)^G\otimes_{A^G}B$, cf.~\cite[Lemma 1]{Knighten}. In particular, $S\simeq (A_{\gotP_i})^{G_i}\simeq (A_\gotP)^{G_i}$
and $R\simeq (A_{\gotP_d})^{G_d}\simeq (A_\gotP)^{G_d}$.
We shall use this repeatedly without further mention.

We proceed in various steps.

(1) In the first step, we prove that $\Phi^{G_d}_{G_i}$ is an isomorphism if $\gotP$ has height one or if $\gotP=(0)$. Furthermore,  we prove that $\Phi^G_H$ is injective and generically an isomorphism for any subgroup $H<G$. 

Some preparations first. 
The extension $R\subset S$ is a Galois-extension of local rings with Galois group $\Gamma$ (see e.g.~\cite[Satz 20.4]{SS}). In particular, $S$ is free over $R=S^\Gamma$ of finite rank;
\begin{equation}\label{Dedekinddifferent}
S\to \Hom_R(S,R):s_1\mapsto (s_2\mapsto \Tr^{\Gamma}(s_1s_2))
\end{equation} is an isomorphism  of
$S$-modules (see \cite[12.5 Korollar]{SS} or \cite[Appendix]{AG}); and the natural inclusion
$S\Gamma\to \End_R(S)$ is an isomorphism of $R$-algebras, where $S\Gamma$ is the twisted group ring.
The twisted group ring $S\Gamma$ is free as  left $S$-module with basis $\{\gamma;\ \gamma\in \Gamma\}$ and multiplication table 
$(s_1\gamma_1)(s_2\gamma_2)=s_1\gamma_1(s_2)\gamma_1\gamma_2$, where $s_1,s_2\in S$,
$\gamma_1,\gamma_2\in\Gamma$. Fix  an $R$-basis $z_1,\ldots, z_n$ of $S$, with dual basis
$z_1^*,\ldots,z_n^*$ of $\Hom_R(S,R)$. There are unique $u_i$'s such that $z_i^*(s)=\Tr^{\Gamma}(u_is)$, for $s\in S$. We get an equality of operators  $\sum_i z_i \Tr^{\Gamma} u_i=1$ in $S\Gamma$.

Put $N:=(A_\gotP\otimes_k M)^{G_i}$. We claim that  the multiplication map
$$\mu:S\otimes_RN^{\Gamma}\to N$$
is an isomorphism of $S\Gamma$-modules, where $S\Gamma$ only acts non-trivially on the first factor of $S\otimes_R N^\Gamma$. Its inverse is
$\alpha:N\to S\otimes_R N^{\Gamma}$ given by the formula
$$\alpha(n):=\sum_i z_i\otimes \Tr^{\Gamma}(u_i n),$$ 
where $n\in N$. We check this as follows. For $n\in N$, we have
$\mu\alpha(n)=\sum_i z_i\Tr^{\Gamma}(u_in)=n$, since $\sum_i z_i\Tr^{\Gamma}u_i=1\in S\Gamma$. Also
$$\alpha(\mu(s\otimes n))=\sum_i z_i\otimes \Tr^{\Gamma}(u_isn)=
\sum_i z_i\otimes \Tr^{\Gamma}(u_is)n=\sum_i z_i\Tr^{\Gamma}(u_is)\otimes n=s\otimes n,$$ 
where $s\in S$, $n\in N^{\Gamma}$.

The localization of $\Phi^{G_d}_{G_i}$ at the prime ideal $\gotP$ is described as follows.
Since
$$\left(\Hom_{A^{G_d}}(A,(A\otimes_k M)^{G_d})\right)_{\gotP_d}\simeq
\Hom_{A^{G_d}_{\gotP_d}}(A_{\gotP_d},\left(A\otimes_k M\right)^{G_d}_{\gotP_d})
\simeq 
\Hom_{R}(A_{\gotP},\left(A_{\gotP}\otimes_k M\right)^{G_d}),$$
and similarly
$\left(\Hom_{A^{G_i}}(A,(A\otimes_k M)^{G_i})\right)_{\gotP_d}\simeq
\Hom_{S}(A_{\gotP},\left(A_{\gotP}\otimes_k M\right)^{G_i})$,
we obtain
$$\left(\Phi^{G_d}_{G_i}\right)_{\gotP}:\ \Hom_{S}(A_\gotP,(A_\gotP\otimes_k M)^{G_i})\to 
\Hom_{R}(A_\gotP,(A_\gotP\otimes_k M)^{G_d}): \phi\mapsto \Tr^{G_d}_{G_i}\circ \phi.$$
Using the isomorphism $\mu$, we finally get the following description
$$\left(\Phi^{G_d}_{G_i}\right)_{\gotP}:\ \Hom_{S}(A_\gotP,S\otimes_R N^\Gamma)\to 
\Hom_{R}(A_\gotP,N^\Gamma): \phi\mapsto \Tr^{\Gamma}\circ \mu\circ \phi.$$

The normality condition on $A$ implies that 
if $\gotP$ has height one, then $R$ and $S$ are both discrete valuation rings.  So $A_\gotP$ is finite
and free over $S$ and $N^\Gamma$ is finite and free over $R$, since both modules are torsion free.
Therefore, to prove that $\Phi^{G_d}_{G_i}$ is an isomorphism at the height one prime ideal $\gotP$,
it suffices to prove that 
$$ \Hom_{S}(S,S)\to 
\Hom_{R}(S,R): \phi\mapsto \Tr^{\Gamma}$$
is an isomorphism. But this follows from the isomorphism in (\ref{Dedekinddifferent}).

A similar proof works for $\gotP=(0)$. In that case, $G=G_d$ and $G_i=1$ and so, with $L$ the quotient field of $A$, we conclude
 $$(\Phi^G)_\gotP: L\otimes_k M\to \Hom_{L^G}(L,(L\otimes_k M)^G)$$
 is an isomorphism. Since for any subgroup $H<G$, also $(\Phi^H)_\gotP$ is an isomorphism
 and $\Phi^G=\Phi^G_H\circ \Phi^H$, it follows that $(\Phi^G_H)_{\gotP}$ is also an isomorphism,
 i.e., $\Phi^G_H$ is generically an isomorphism.
 
 Finally, we prove injectivity of $\Phi^G_H$. Let $\phi:A\to (A\otimes_k M)^H$ be an $A^H$-homomorphism,
 such that $\Tr_H^G\circ \phi(a)=0$, for all $a\in A$. We extend $\phi$ to an $L^H$-linear map
 $\tilde{\phi}:L\to (L\otimes_k M)^H$ by $\tilde{\phi}(\frac{a}{s})=\frac{1}{s}\phi(a)$, where $a\in A$, $s\in A^H$, $s\neq 0$. So $\tilde{\phi}$ is in the kernel of the map
 $$\Hom_{L^H}(L,(L\otimes_k M)^H)\to \Hom_{L^G}(L,(L\otimes_k M)^G):\psi\to \Tr^G_H\circ \psi,$$
which we just proved to be an isomorphism. So $\tilde{\phi}=0$, hence $\phi=0$ (since 
$(A\otimes_kM)^H$ is torsion free over $A^H$) and so $\Phi^G_H$ is injective.

(2) In the second step, we show that $\Phi^G_{G_d}$ is an isomorphism at $\gotP$ (without any restriction on the height of $\gotP$), using completion as a tool. 

Some preparations first.
For any $A^G$-module $N$, write $\widehat{N}=N^{\wedge}$ for the completion of the localization $N_\gotp$  with respect to the $\gotp$-adic topology. 
So $\widehat{N}$ is a module over the complete local ring $\widehat{A^G}$, the completion of 
the local ring $A^G_\gotp$ at its maximal ideal. We recall the following basic facts about
completion. Putting the hat on is an exact functor from
$A^G$-modules to $\widehat{A^G}$-modules, and so $\widehat{A}$ is flat over $\widehat{A^G}\simeq \widehat{A}^G$. The ring $\widehat{A}$ is a complete semilocal ring, whose maximal ideals correspond to the prime ideals of $A$ in the $G$-orbit of $\gotP$, see~\cite[Corollary 7.6]{Eisenbud}. 
More explicitly,  there is a collection of  idempotents $\{e_{g\gotP}, gG_d\in G/G_d\}$ in $\widehat{A}$,  such that $g(e_\gotP)=e_{g\gotP}$ for $g\in G$ and
$$e_{g\gotP}^2=e_{g\gotP},\ 
e_{g\gotP}\cdot e_{g'\gotP}=0 \hbox{  if $g\gotP\neq g'\gotP$},\  1=\sum_{gG_d\in G/G_d} e_{g\gotP}\in \widehat{A}.$$
The Cartesian product alluded to before is then $\widehat{A}=\bigoplus_{gG_d\in G/G_d}\widehat{A}e_{g\gotP}$. The maximal ideal of $\widehat{A}$ corresponding to $\gotP$ is 
 $$\gotm:=\widehat{\gotp}\widehat{A}e_\gotP\oplus \widehat{A}(1-e_\gotP)=
\widehat{\gotP}=\gotP\widehat{A},$$ and $g(\gotm)$ is the maximal
ideal corresponding to $g\gotP$, for $g\in G$; and these are all the maximal ideals of $\widehat{A}$.
Then $\widehat{A}e_\gotP\simeq\widehat{A}_\gotm\simeq \widehat{A_\gotP}$ is isomorphic to the completion of $A_\gotP$
at its maximal ideal. Similarly, for any 
 $A$-module $N$ we have that $e_\gotP \widehat{N}\simeq\widehat{N}_\gotm\simeq 
 \widehat{N_\gotP}$
is isomorphic to the completion of $N_\gotP$ at the maximal ideal of $A_\gotP$.

 We apply all this to the $A$-module $$N:=\Hom_{A^{G_d}}(A,(A\otimes_kM)^{G_d}).$$
First of all, we can identify $\widehat{N}$ with $\Hom_{\widehat{A}^{G_d}}(\widehat{A},(\widehat{A}\otimes_kM)^{G_d}).$ Let $\phi\in \widehat{N}$. By definition of the $\widehat{A}$-action on  $\widehat{N}$, the element $e_\gotP\phi\in e_\gotP\widehat{N}$ is the
 $\widehat{A}^{G_d}$-morphism  $(e_\gotP\phi)(a):=\phi(ae_\gotP)$, where $a\in\widehat{A}$.
Since $e_\gotP$ is $G_d$-invariant and $\phi$ is $\widehat{A}^{G_d}$-linear, we also have
$$(e_\gotP\phi)(a)=\phi(ae_\gotP)=\phi(ae_\gotP e_\gotP)=e_\gotP\phi(ae_\gotP);$$ 
and so the image of $e_\gotP\phi$ is contained in $(\widehat{A}e_\gotP\otimes_kM)^{G_d}$.
Therefore, we can make the identification
 $$e_\gotP \widehat{N}\simeq \Hom_{\widehat{A}^{G_d}}(\widehat{A}e_\gotP,(\widehat{A}e_\gotP\otimes_kM)^{G_d}).$$
 
 Similarly, we identify
 $$e_\gotP \left(\Hom_{A^{G}}(A,(A\otimes_kM)^{G}) \right)^{\wedge}\simeq \Hom_{\widehat{A}^{G_d}}(\widehat{A}e_\gotP,(\widehat{A}\otimes_kM)^{G_d}),$$
 and then identify  $e_\gotP\cdot \left(\Phi^G_{G_d}\right)^{\wedge}$ with the map
\begin{equation}\label{iso}
 \Hom_{\widehat{A}^{G_d}}(\widehat{A}e_\gotP,(\widehat{A}e_\gotP\otimes_kM)^{G_d})\to
  \Hom_{\widehat{A}^G}(\widehat{A}e_\gotP,(\widehat{A}\otimes_kM)^{G}):\
  \phi\mapsto \Tr^{G}_{G_d}\circ \phi.
 \end{equation}
 We want to show that this map is an isomorphism.
 
As a preparation of the proof, we first remark that,  after $\gotp$-adic completion, the natural map $A^G_\gotp\to (A_\gotP)^{G_d}$ can be identified 
with  $$\pi:\widehat{A}^G\to (\widehat{A} e_{\gotP})^{G_d};\ \pi(a):=ae_\gotP.$$
We claim that $\pi$ is an isomorphism of algebras, with
inverse map given by the relative trace map $\Tr_{G_d}^G$.
This is seen as follows. 
Let $a\in \widehat{A}$ and $\sigma\in G$, then we can write
$$a=\sum_{gG_d\in G/G_d} ae_{g\gotP}\ \hbox{ and so, }
\sigma(a)=\sum_{gG_d\in G/G_d} \sigma(a)e_{\sigma g\gotP}.$$
Comparing, we get that $a \in \widehat{A}^G$ {\em if and only if} $ae_\gotP\in (\widehat{A}e_\gotP)^{G_d}$
and $ae_{g\gotP}=g(ae_{\gotP})$ for all $gG_d\in G/G_d$  {\em if and only if}
$a=\Tr^G_{G_d} (ae_\gotP)=\Tr^G_{G_d} \pi(a).$ Hence, indeed, $\pi$ is an algebra  isomorphism with inverse
$\Tr^G_{G_d}.$

Similarly, after $\gotp$-adic completion, the natural map  $(A_\gotp\otimes_kM)^G\to (A_\gotP\otimes_k M)^{G_d}$  can be identified with the map $$\pi_M:(\widehat{A}\otimes_kM)^G\to (\widehat{A} e_{\gotP}\otimes_kM)^{G_d}$$ induced by projection, and,
again, $\pi_M$ is an isomorphism with inverse map $\Tr^G_{G_d}$.

Let now $\psi\in  \Hom_{\widehat{A}^G}(\widehat{A}e_\gotP,(\widehat{A}\otimes_kM)^{G})$ and consider
the composition
$$\pi_M\circ \psi: \widehat{A}e_\gotP\to (\widehat{A}e_\gotP\otimes_kM)^{G}).$$
Let $b\in  \widehat{A}^{G_d}$ and $ae_\gotP\in\widehat{A}e_\gotP$. 
Then $\Tr^{G}_{G_d}(be_\gotP)\in \widehat{A}^G$ and so
$$\pi_M\circ \psi(bae_\gotP)=
e_\gotP \psi(\Tr^{G}_{G_d}(be_\gotP) ae_\gotP)=
e_\gotP\Tr^{G}_{G_d}(be_\gotP)\psi(ae_\gotP)=b(\pi_M\circ \psi)(bae_\gotP),$$
using the orthogonality of the idempotents and the $\overline{A}^G$-linearity of $\psi$. We conclude that
$\pi_M\circ \psi$ is $\widehat{A}^{G_d}$-linear. We get, therefore, a well-defined map
$$\Psi:  \Hom_{\widehat{A}^G}(\widehat{A}e_\gotP,(\widehat{A}\otimes_kM)^{G})\to
 \Hom_{\widehat{A}^{G_d}}(\widehat{A}e_\gotP,(\widehat{A}e_\gotP\otimes_kM)^{G_d})
:\
  \psi\mapsto \pi_M\circ \psi,
 $$
 which is the inverse of the map $e_\gotP\cdot \left(\Phi^G_{G_d}\right)^\wedge$ in (\ref{iso}), since
 $\Tr^{G}_{G_d}\circ \pi_M$ and $\pi_M\circ \Tr^{G}_{G_d}$ are the identity maps.
 
 Since $e_\gotP\cdot \left(\Phi^G_{G_d}\right)^\wedge$ can be identified with what we get from $\Phi^G_{G_d}$
 after localizing at $\gotP$ and then completing at $\gotP$, we conclude that, 
indeed, $\Phi^G_{G_d}$ becomes an isomorphism after localizing at $\gotP$ and then completing with respect to the $\gotP$-adic
topology. But then $\Phi^G_{G_d}$ already becomes an isomorphism after localizing at
the prime ideal $\gotP$, see \cite[p.~203]{Eisenbud}. Which is what we wanted to show.

(3) In the last step, we complete the proof of the theorem.
That $\Phi_H^G$ is indeed a homomorphism of reflexive $A$-modules is implied by Lemma~\ref{reflexive}(ii), (iv). 

Combining  (1) and (2), we get that $\Phi^G_{G_i}$ is an isomorphism at every prime ideal $\gotP$ of  height one.
Suppose $H<G$ is a subgroup and $\gotP\subset A$ a prime ideal of height one, such that 
the inertia subgroups coincide $H_i(\gotP)=G_i(\gotP)$. Since 
$$\Phi_H^G\circ\Phi^H_{H_i(\gotP)}=\Phi^G_{G_i(\gotP)}$$ and
both $\Phi^H_{H_i(\gotP)}$ and $\Phi^G_{G_i(\gotP)}$ are isomorphisms at $\gotP$ by (2), it follows
that $\Phi_H^G$ is an isomorphism at $\gotP$. Hence (i).
And from this also the isomorphism $\left(C^G_H\right)_\gotP\simeq
\left(C^{G_i}_{H_i}\right)_\gotP$ follows, hence (iii).
The condition in (ii) implies that $G_i(\gotP)=K_i(\gotP)$
for all height one prime ideals. So $\Phi_K^G$ is a pseudo-isomorphism between
reflexive modules and thus is an isomorphism, by Lemma~\ref{reflexive}. Hence (ii).
\end{proof}

\subsection{Proof of Theorem~\ref{BCB}}
The following is a generalization of a result of Benson and Crawley-Boevey~\cite[Corollary 3.12.2]{Benson1993} (we formulated this result earlier as Theorem~\ref{BCB}).

\begin{corollary}\label{newBenson}
Let $A$ be a normal graded algebra with a finite group $G$ of $k$-algebra automorphisms,
and $M$ a finite dimensional $kG$-module.

(i) 
For any subgroup $K\leq G$ such that $K\supseteq W$ we have
$$s_{A^G}(A^G(M))=s_{A^K}(A^K(M)).$$
In particular, if $W$ acts trivially on $M$, then 
$$s_{A^G}(A^G(M))=0.$$

(ii)
Suppose $G_i(\gotP)\cap G_i(\gotP')=\{1\}$ for distinct height one prime ideals $\gotP$ and
$\gotP'$ of $A$.  Then
$$s_{A^G}(A^G(M))=\sum_{\gotP}s_{A^{G_i(\gotP)}}(A^{G_i(\gotP)}(M))$$
where the sum is over the homogeneous height one prime ideals of $A$.
\end{corollary}

\begin{proof}
(i) From Theorem~\ref{main}, we get an isomorphism of $A$-modules
$$\Hom_{A^K}(A,A^K(M))\simeq \Hom_{A^G}(A,A^G(M)).$$
We calculate using Proposition~\ref{Benson1}:
$$s_{A^G}(\Hom_{A^G}(A,A^G(M)))=|G|s_{A^G}(A^G(M))-\dim_kM\cdot s_{A^G}(A),$$
since $s_{A^G}(\Ext_{A^G}^1(A,A^G(M)))=0$ (because $A$ is $A^G$-torsion free and $A^G$ is normal).
On the other hand, Lemma~\ref{degpsi} and Proposition~\ref{Benson1} give
\begin{eqnarray*}
\lefteqn{s_{A^G}(\Hom_{A^K}(A,A^K(M)))=}\\
&=& s_{A^K}(\Hom_{A^K}(A,A^K(M)))\frac{|G|}{|K|}+|K|\dim_kM\cdot s_{A^G}(A^K)\\
&=&|G|s_{A^K}(A^K(M))-\dim_kM\cdot s_{A^K}(A)\frac{|G|}{|K|},
\end{eqnarray*}
because $s_{A^G}(A^K)=0$, by Proposition~\ref{Benson2}.
Since 
$$s_{A^G}(A)=s_{A^K}(A)\frac{|G|}{|K|}+|K|s_{A^G}(A^K)=s_{A^K}(A)\frac{|G|}{|K|},$$
we get that
$$s_{A^G}(A^G(M))=s_{A^K}(A^K(M)).$$

(ii) Let $\gotP$ be a homogeneous height one prime ideal of $A$ with inertia subgroup $H:=G_i(\gotP)$.
Let $\gotP'$ be another homogeneous height one prime ideal of $A$. Then
by the assumption $H_i(\gotP')=1$. So the cokernel $C_{\gotP}$ of 
$$\Phi^H:A\otimes_k M\to \Hom_{A^H}(A,A^H(M))$$ 
vanishes at $\gotP'$. Therefore, since $\Phi^H$ is injective, we get
$$\length_{\gotP}(C_\gotP)\psi(A/\gotP)=\psi(\Hom_{A^H}(A,A^H(M)))-\psi(A)\dim_kM.$$
So, from the theorem, it follows that
\begin{eqnarray*}
\lefteqn{\psi(\Hom_{A^G}(A,A^G(M)))-\psi(A)\dim_kM=}\\
&=& \sum_{\gotP}\left(
\psi(\Hom_{A^{G_i(\gotP)}}(A,A^{G_i(\gotP)}(M)))-\psi(A)\dim_kM
\right),
\end{eqnarray*}
where  the sum is over the homogeneous height one prime ideals of $A$.

A direct calculation as in (i) shows that
$$\psi(\Hom_{A^G}(A,A^G(M)))-\psi(A)\dim_kM=-\deg(A)s_{A^G}(A^G(M))+2\deg(A)\dim_kM \frac{s_{A^G}(A)}{|G|}$$
and similarly, where $G$ is replaced by $G_i(\gotP)$. Using Proposition~\ref{Benson2}(ii) that
$$\frac{1}{|G|}s_{A^G}(A)=\sum_{\gotP}\frac{1}{|G_i(\gotP)|}s_{A^{G_i(\gotP)}}(A),$$
we get
$$s_{A^G}(A^G(M))=
\sum_{\gotP}s_{A^{G_i(\gotP)}}(A^{G_i(\gotP)}(M))
,$$
where both sums are over the homogeneous height one prime ideals of $A$.
\end{proof}

The following corollary (of the proof of Theorem~\ref{main}) is used in the proof of Proposition~\ref{pointstabilizer}.
\begin{lemma}\label{hulp}
Suppose $A^G(M)$ is free (or Cohen-Macaulay), and $\gotP\subset A$ a prime
ideal with decomposition group $G_d$ and inertia group $G_i$. Then

(i) $A^{G_d}(M)$ is  free (or Cohen-Macaulay)
at the prime ideal $\gotP\cap k[V]^{G_d}$.

(ii) $A^{G_i}(M)$ is  free (or Cohen-Macaulay)
at the prime ideal $\gotP\cap k[V]^{G_i}$.
\end{lemma}

\begin{proof}
We use the results, techniques and notation of the proof of Theorem~\ref{main}.

(i) We proved that
$$(\widehat{A}e_\gotP\otimes_kM)^{G_d}\simeq (\widehat{A}\otimes_k M)^G\simeq ((A\otimes_k M)^G)^\wedge,$$
and so $(\widehat{A}e_\gotP\otimes_kM)^{G_d}$ is free (or Cohen-Macaulay) 
 over $\widehat{A}^G\simeq (\widehat{A}e_\gotP)^{G_d}$.
 On the other hand, $\widehat{A}e_\gotP$ is also the completion of $A_{\gotP_d}$ with respect to the
 $\gotP_d$-adic topology. So the completion of $A^{G_d}(M)_{\gotP_d}$ with respect to the $\gotP_d$-adic
 topology is free (or Cohen-Macaulay), hence $A^{G_d}(M)_{\gotP_d}$ is free (or Cohen-Macaulay)
 over $A^{G_d}_{\gotP_d}$.
 
 (ii) We found that  multiplication induces an isomorphism of 
 $S\Gamma$-modules $$S\otimes_R(A_\gotP\otimes_k M)^{G_d}\simeq (A_\gotP\otimes_k M)^{G_i}.$$
 Since we proved in (i) that $(A_\gotP\otimes_k M)^{G_d}$ is free  over $R$ (or Cohen-Macaulay) it follows  that   $(A_\gotP\otimes_k M)^{G_i}$ is also free (or Cohen-Macaulay), since $S$ is free over $R$).
 \end{proof}
 
 \begin{remark}\label{psi}
Reformulating the corollary in terms of the numerical invariant $\psi$ we get
for any subgroup $W\leq K\leq G$ that
$$|G|\psi(A^G(M))=|K|\psi(A^K(M)).$$
In particular, when $W=1$, then 
$$|G|\psi(A^G(M))=\psi(A)\dim_kM.$$
And if $G_i(\gotP)\cap G_i(\gotP')=\{1\}$, for distinct height one prime ideals $\gotP$ and
$\gotP'$ of $A$, then 
$$|G|\psi(A^G(M))-\dim_k M \psi(A)=
\sum_{\gotP}\left(|G_i(\gotP)|\psi(A^{G_i(\gotP)}(M))-\dim_k M \psi(A)\right),$$
where the sum is over the homogeneous height one prime ideals of $A$.
\end{remark}

\begin{example}
Take the special case where $M=kG$ is the regular representation. If we restrict this representation
to a subgroup $H$, it decomposes as the direct sum of $\frac{|G|}{|H|}$ copies of the
regular representation $kH$ of $H$. So under the hypothesis of Corollary~\ref{newBenson}(ii)  we get
\begin{eqnarray*}
s_{A^G}(A)&=&s_{A^G}(A^G(kG))\\
&=&\sum_{\gotP}s_{A^{G_i(\gotP)}}(A^{G_i(\gotP)}(kG))\\
&=&\sum_{\gotP}\frac{|G|}{|G_i(\gotP)|} s_{A^{G_i(\gotP)}}(A^{G_i(\gotP)}(kG_i(\gotP)))\\
&=&\sum_{\gotP}\frac{|G|}{|G_i(\gotP)|} s_{A^{G_i(\gotP)}}(A)\\
\end{eqnarray*}
where the sum is over the homogeneous height one prime ideals of $A$.
Hence we recover Proposition~\ref{Benson2}(ii). 
\end{example}

\subsection{Proof of the Jacobian criterion}\label{proofjacobian}
We shall use the techniques of the proof of Theorem~\ref{main} to give a proof of the
Jacobian criterion of freeness of modules of covariants, i.e., Theorem~\ref{Jacobian}.

\begin{proof}[Proof of Theorem~\ref{Jacobian}]
Write $A:=k[V]$. Since $\Jac_M$ is multilinear in its arguments, we get an $A^G$-linear map 
$$\Jac_M: \wedge^m_{A^G}\left(A\otimes_k M\right)^G\to A:\
\omega_1\wedge\omega_2\wedge\ldots\wedge \omega_m\mapsto
\Jac_M(\omega_1,\ldots,\omega_m).$$
Its image is just $J^G_M$, an $A^G$-submodule of $A$ contained in $A^G_\lambda$.
Let $A^G\subset B$ be a flat extension of algebras, then
\begin{eqnarray*}
B\otimes_{A^G}\wedge^m_{A^G}\left(A\otimes_k M\right)^G
&\simeq&\wedge^m_{B}\left(B\otimes_{A^G}\left(A\otimes_k M\right)^G\right)\\
&\simeq&\wedge^m_{B}\left(B\otimes_{A^G}A\otimes_k M\right)^G,
\end{eqnarray*}
by \cite[p.~571]{Eisenbud} and our remark at the beginning of the proof of Theorem~\ref{main}.
The corresponding Jacobian map
$$B\otimes_{A^G}\Jac_M: \wedge^m_{B}\left(B\otimes_{A^G}A\otimes_k M\right)^G\to B\otimes_{A^G}A$$
has then image $B\otimes_{A^G}J^G_M\subset B\otimes_{A^G}A$.

Let $\gotP\subset A$ be a prime ideal. We use the results, techniques and notation of the proof of Theorem~\ref{main}. In particular, $\gotp$, $G_i$, $G_d$, $\widehat{A}$, $e_\gotP$, $\gotm$, etc. We will break the proof in several steps. That the action on $A:=k[V]$ comes from
a linear action on $V$ does not play a role in the first three steps.

(1) In the first step, we prove that $J^{G_d}_MA_\gotP^{G_i}=J^{G_i}_MA_\gotP^{G_i}$.
Multiplication induces an isomorphism
$$A_\gotP^{G_i}\otimes_{A_{\gotP}^{G_d}} (A_{\gotP}\otimes_k M)^{G_d}\simeq (A_{\gotP}\otimes_k M)^{G_i},$$
and $A_\gotP^{G_d}\subset A_\gotP^{G_i}$ is a free extension. So from the remarks above,
it follows that 
$$A_\gotP^{G_i}\otimes_{A^{G_i}}J^{G_i}_M=A_\gotP^{G_i}\otimes_{A^{G_d}_\gotP}A_\gotP^{G_d}\otimes_{A^{G_d}}J^{G_d}_M,$$
or $J^{G_d}_MA_\gotP^{G_i}=J^{G_i}_MA_\gotP^{G_i}$.

(2) In the second step, we prove that $J^{G}_M\widehat{A_\gotP}^{G_d}=J^{G_d}_M\widehat{A_\gotP}^{G_d}$.
Projection induces isomorphisms $\widehat{A}^G\simeq \widehat{A}^{G_d}e_\gotP(\simeq
\widehat{A_\gotP}^{G_d})$ and
$$(\widehat{A}\otimes_k M)^G\simeq (\widehat{A}e_\gotP\otimes_k M)^{G_d}
(\simeq (\widehat{A_\gotP}\otimes_k M)^{G_d}),$$ with inverse
$\Tr^G_{G_d}$.
In particular, if  $\omega\in\left(\widehat{A}\otimes_k M\right)^G$, then $e_\gotP\omega\in \left(\widehat{A}\otimes_k M\right)^{G_d}$ and $\omega=\Tr^G_{G_d}(e_\gotP\omega)$.

Consider $\omega_1\wedge\ldots\wedge \omega_m\in \wedge^m_{\widehat{A}}\left(\widehat{A}\otimes_k M\right)$, where $\omega_1,\ldots,\omega_m$ are elements of $\left(\widehat{A}\otimes_k M\right)^G$. From the orthogonality of the idempotents, it follows
 \begin{eqnarray*}
 \omega_1\wedge\ldots\wedge \omega_m&=&
 (\Tr^G_{G_d}(e_\gotP\omega_1))\wedge\ldots\wedge (\Tr^G_{G_d}(e_\gotP\omega_m))\\
 &=& \Tr^G_{G_d}\left(e_\gotP \omega_1\wedge\ldots\wedge e_\gotP\omega_m\right),
 \end{eqnarray*}
 or, by definition of the Jacobian determinant (with $v_1,\ldots,v_m$ the fixed basis of $M$),
 \begin{eqnarray*}
 \Jac_M\left(\omega_1,\ldots, \omega_m\right)\otimes (v_1\wedge\ldots\wedge v_m)&=&
 \Tr^G_{G_d}\left(\Jac_M\left(e_\gotP\omega_1,\ldots, e_\gotP\omega_m\right)\otimes (v_1\wedge\ldots\wedge v_m)\right)\\
 &=&\Tr^G_{G_d}\left(e_\gotP\Jac_M\left(\omega_1,\ldots, \omega_m\right)\otimes (v_1\wedge\ldots\wedge v_m)\right)
 \end{eqnarray*}
 in $\widehat{A}\otimes_k\wedge^mM$. 
 Therefore, $\Jac_M\left(\omega_1\wedge\ldots\wedge \omega_m\right)e_\gotP=
  \Jac_M\left(e_\gotP\omega_1\wedge\ldots\wedge e_\gotP\omega_m\right).$
  We conclude that 
  $J^G_M\widehat{A}^Ge_\gotP$ equals the image of the Jacobian map 
  $$\wedge^m_{\widehat{A}^{G_d}}\left(\widehat{A}e_\gotP\otimes_k M\right)^{G_d}\to \widehat{A}e_\gotP,$$
i.e., 
$$J^G_M\widehat{A}^Ge_\gotP=J^{G_d}_M\widehat{A}^{G_d}e_\gotP=J^{G_d}_M\widehat{A}^{G_d}e_\gotP.$$ So indeed,
 $J^G_M\widehat{A_\gotP}^{G_d}=J^{G_d}_M\widehat{A_\gotP}^{G_d}$.

(3) In this step, we prove that if $\gotP$ contains $J^G_M$, 
then the inertia subgroup of $\gotP$ is non-trivial. 
Since $J^G_M\subseteq \gotP$, it follows from the second step that
$J_M^{G_d}\widehat{A_\gotP}^{G_d}=J_M^{G}\widehat{A_\gotP}^{G_d}\subseteq \gotP\widehat{A_\gotP}^{G_d}$. Since, by the first step,
$J_M^{G_i}A_\gotP^{G_i}=J_M^{G_d}A_\gotP^{G_i}$,
we obtain
$$J_M^{G_i}\widehat{A_\gotP}^{G_i}=J_M^{G_d}\widehat{A_\gotP}^{G_i}=J_M^{G}\widehat{A_\gotP}^{G_i}\subseteq \gotP\widehat{A_\gotP}^{G_i}.$$
Supposing now that the inertia subgroup is trivial, then $J_M^{G_i}=A$ and we conclude that
$\widehat{A_\gotP}\subseteq \gotP\widehat{A_\gotP}$, which is a contradiction. So $G_i$ is non-trivial.

(4) In this step, we prove (i) and (ii). Let $F_{G,M}$ be a greatest common divisor of all elements on $J^G_M$. It is unique
up to a scalar and $k[V]F_{G,M}$ is the intersection of all height one prime ideals in $k[V]$
containing $J^G_M$, and so $J^G_M\subseteq k[V]F_{G,M}$. Let $f$ be an irreducible factor of 
$F_{G,M}$ generating the height one prime ideal $\gotP$. The multiplicity $\mu$ of $f$ in $F_{G,M}$ then coincides
with the integer $\mu $ such that $J^G_MA_\gotP=(\gotP A_{\gotP})^{\mu}$ in the discrete
valuation ring $A_\gotP$. It also coincides 
with the integer $\mu$ such that $J^G_M\widehat{A_\gotP}=(\gotP \widehat{A_{\gotP}})^{\mu}$ in the discrete
valuation ring $\widehat{A_\gotP}$. Since we showed in the first two steps  that
$J^G_M\widehat{A_\gotP}=J^{G_i}_M\widehat{A_\gotP}$,
the multiplicity of $f$ in $F_{G,M}$ equals  the multiplicity of $f$ in $F_{G_i,M}.$

Since $J^G_M\subseteq \gotP$, 
the inertia subgroup of $\gotP$ is non-trivial by the third step.
Since we are dealing with a linear action,  this forces $f$ to be a linear form, say $f=x_U$,
where $x_U$ defines a linear subspace $U\subset V$ of codimension one. Then $G_i$ identifies
with the point-stabilizer $G_U$ of $U$, consisting of reflections having $U$ as reflection hyperplan. By Proposition~\ref{pointstabilizer},  $k[V]^{G_U}(M)$ is free, say with basis
$\omega_1,\ldots,\omega_m$. Then $J^{G_U}_M$ is generated by
$F_{G_U,M}=\Jac_M(\omega_1,\ldots,\omega_m)$, having  necessarily degree
$e_U(M)$ (the sum of the degrees of the $\omega_i$'s). Since the intersection of $G_i$ with the inertia subgroup of any other height
one prime ideal is trivial, it follows that $x_U$ is the only irreducible factor of $h$, hence
up to a scalar we get $F_{G_U,M}=x_U^{e_U(M)}$.

We conclude that ${e_U(M)}$ is also the multiplicity of $x_U$ in $F_{G,M}$.
Comparing with the definition of $F_M$ given  before the statement of the theorem, we conclude
that $F_M=F_{G,M}$ up to a non-zero scalar. This shows (i) and (ii).

(5) Finally, we prove (iii). Let $\omega_1,\ldots,\omega_m\in ([k[V]\otimes_kM)^G$.
If (a) holds, then $J^G_M$ is generated by $\Jac_M(\omega_1,\ldots,\omega_m)$, and
so by (ii), is equal to $F_M$ up to a non-zero scalar. Hence (b).
We remark that $s_{k[V]^G}(k[V]^G(M))$ is equal to the degree of $F_M$, hence
(b) and (c) are clearly equivalent.

Suppose (b). Then $F_M\in J^G_M$ and in particular, $F_M\in k[V]^G_\lambda$. 
Let $h\in J^G_M$. Then $h\in k[V]^G_\lambda$ and by (ii) there exists a $b\in k[V]$ such
that $h=bF_M$ and so $b\in k[V]^G$, i.e. $J_M^G=k[V]^GF_M$.

A square matrix with coefficients in a field is invertible {\em if and only if} its determinant is non-zero. Hence, if $v_1,\ldots, v_m$ are  vectors in an $m$-dimensional vector space, 
they form  a basis {\em if and only if} $v_1\wedge \ldots \wedge v_m$ is
non-zero in the top exterior power of the vector space.  Let now $L$ be the quotient field of $k[V]$. 
Since $\Jac_M(\omega_1,\ldots,\omega_m)$ is non-zero, the
$\omega_1,\ldots,\omega_m$ form a basis over $L^G$ of $(L\otimes_k M)^G$.
Let $\omega\in k[V]^G(M)$ be non-zero. Then it follows that there are $b,b_1,\ldots,b_m\in k[V]^G$ such that
$$b\omega=\sum_ib_i\omega_i.$$ 
There are $b_i'\in k[V]^G$ such that
$$\Jac_M(\omega_1,\ldots,\omega_{i-1},b\omega,\omega_{i+1},\ldots,\omega_m)=bb_i'F_M.$$
On the other hand, using $b\omega=\sum_ib_i\omega_i$, we get
$$b_i F_M=bb_i'F_M.$$
So $b_i=bb_i'$ and $b\omega=b\sum_ib_i'\omega_i$. Since $k[V]^G(M)$ is torsion free, we
get 
$$\omega=\sum_ib_i'\omega_i,$$ 
and so $\omega_1,\ldots,\omega_m$ generate
the module of covariants of rank $m$, hence (a). This finishes the proof of (iii).
\end{proof}

\end{document}